\pdfoutput=1
\documentclass{amsart}

\usepackage[T1]{fontenc}
\usepackage[utf8]{inputenc}
\usepackage{amsmath,amssymb,amsthm}
\usepackage{enumitem}
\usepackage{booktabs}
\usepackage[colorlinks=true,linkcolor=blue,citecolor=blue,urlcolor=blue]{hyperref}

\newtheorem{theorem}{Theorem}[section]
\newtheorem{proposition}[theorem]{Proposition}
\newtheorem{lemma}[theorem]{Lemma}
\newtheorem{corollary}[theorem]{Corollary}

\theoremstyle{definition}
\newtheorem{definition}[theorem]{Definition}
\newtheorem{example}[theorem]{Example}
\newtheorem{remark}[theorem]{Remark}
\newtheorem{question}[theorem]{Question}

\usepackage{needspace}
\newcommand{\OO}{\mathcal{O}}
\newcommand{\EE}{\mathcal{E}}
\newcommand{\FF}{\mathcal{F}}
\newcommand{\II}{\mathcal{I}}
\newcommand{\CC}{\mathbb{C}}
\newcommand{\PP}{\mathbb{P}}
\newcommand{\ZZ}{\mathbb{Z}}
\newcommand{\Hilb}{\operatorname{Hilb}}
\newcommand{\Num}{\operatorname{Num}}
\newcommand{\Pic}{\operatorname{Pic}}
\newcommand{\Ext}{\operatorname{Ext}}
\newcommand{\Hom}{\operatorname{Hom}}
\newcommand{\End}{\operatorname{End}}
\newcommand{\Rep}{\operatorname{Rep}}
\newcommand{\rk}{\operatorname{rk}}
\newcommand{\len}{\ell}
\newcommand{\QQ}{\mathbb{Q}}
\newcommand{\RR}{\mathbb{R}}
\newcommand{\NE}{\overline{\operatorname{NE}}}
\newcommand{\Eff}{\operatorname{Eff}}

\title[A no-go theorem for special Ulrich bundles]{A no-go theorem for special Ulrich bundles,\\ with a complement on primary Burniat surfaces}

\author{Cristian Anghel}
\address{``Simion Stoilow'' Institute of Mathematics of the Romanian Academy, P.O. Box 1-764, RO-014700 Bucharest, Romania}
\email{cristian.anghel@imar.ro}

\author{Filip Chindea}
\address{``Simion Stoilow'' Institute of Mathematics of the Romanian Academy, P.O. Box 1-764, RO-014700 Bucharest, Romania}
\email{filip.chindea@imar.ro}

\subjclass[2020]{Primary 14J60; Secondary 14J29, 13C14}
\keywords{Ulrich bundle, Cayley--Bacharach property, special polarization, Burniat surface, bidouble cover}

\begin{document}

\begin{abstract}
Let $X$ be a smooth projective surface with $p_g=0$, and let $H$ be an ample divisor with $h^0(\OO_X(H))\neq0$, $\chi(\OO_X(H))\ge q$, and $h^1(\OO_X(H))\neq0$. We prove that no rank two bundle $\EE$ with $c_1(\EE)=3H+K_X$, with the Ulrich value of $c_2(\EE)$, and satisfying $h^0(\EE(-H))=0$, can arise from an extension
\[
0\longrightarrow \OO_X(H+K_X)\longrightarrow \EE\longrightarrow \OO_X(2H)\otimes\II_Z\longrightarrow0.
\]
Thus, for this natural Cayley--Bacharach construction, non-speciality of the polarization is necessary rather than merely convenient.

We then study primary Burniat surfaces. We show that every ample and base point free divisor is non-special; consequently, every polarization carries a stable special Ulrich bundle of rank two, and the surface is strictly Ulrich wild with respect to every polarization. We also locate the special ample classes on three numerical rays through $K_X$, compute explicit families on these rays, and analyze a twisted-kernel variant of the construction. The degree bound underlying the non-speciality result overlaps with recent work of Y. Cho, while the global consequences and the no-go theorem are independent.
\end{abstract}

\maketitle

\section{Introduction}\label{sec:intro}

\subsection{Ulrich bundles}\label{subsec:ulrich}
Let $X\subseteq\PP^N$ be a smooth projective variety of dimension $n$ over $\CC$, polarized by the very ample divisor $H$ with $\OO_X(H)=\OO_{\PP^N}(1)|_X$.

\begin{definition}\label{def:ulrich}
A vector bundle $\EE$ on $X$ is an \emph{Ulrich bundle} (with respect to $H$) if $H^\bullet(X,\EE(-iH))=0$ for $i=1,\dots,n$.
\end{definition}

For a surface this is the six vanishings
\begin{equation}\label{eq:ulrich-vanishing}
h^j\bigl(X,\EE(-H)\bigr)=h^j\bigl(X,\EE(-2H)\bigr)=0,\qquad j=0,1,2 .
\end{equation}
Equivalently, $\pi_*\EE\cong\OO_{\PP^n}^{\oplus d\cdot\rk\EE}$ for a finite linear projection $\pi:X\to\PP^n$, $d=\deg X$, or again $\bigoplus_mH^0(\EE(mH))$ is a maximally generated maximal Cohen--Macaulay module. Existence on every projective variety was raised as a question by Eisenbud and Schreyer \cite{ES03}; see \cite{Bea18,Cos17}. The Riemann--Roch constraints $\chi(\EE(-H))=\chi(\EE(-2H))=0$ give (cf.\ \cite[\S2]{Cas17})
\begin{gather}
c_1(\EE)\cdot H=\frac{\rk\EE}{2}\bigl(3H^2+H\cdot K_X\bigr),\label{eq:chern}\\
c_2(\EE)=\tfrac12\bigl(c_1^2-c_1\cdot K_X\bigr)-\rk(\EE)\bigl(H^2-\chi(\OO_X)\bigr).\notag
\end{gather}
Following \cite{ES03,Cas17}, a rank two Ulrich bundle is \emph{special} if $c_1(\EE)=3H+K_X$, an equality in $\Pic(X)$ and not merely in $\Num(X)$ --- Proposition~\ref{prop:numkernel} deals with the numerical variant. Then \eqref{eq:chern} forces
\begin{equation}\label{eq:c2special}
c_2(\EE)=\tfrac12\bigl(5H^2+3H\cdot K_X\bigr)+2\chi(\OO_X).
\end{equation}

\subsection{Known results, and two constructions from 1994}
On curves Ulrich line bundles always exist \cite{ES03}. In dimension two the question is open in general, though a rich list of cases is known: anticanonical rational surfaces, K3 surfaces \cite{AFO17,Fae19}, abelian surfaces \cite{Bea16}, Enriques surfaces \cite{BN18}, surfaces with $p_g=q=0$ and \emph{non-special} very ample polarization \cite{Cas17} or with $|H-K_X|$ containing an irreducible curve \cite{Bea16b}, extended to regular surfaces of non-negative Kodaira dimension in \cite{Cas22}, non-special surfaces with $p_g=0$, $q=1$ \cite{Cas19}, certain irregular surfaces \cite{Lop21}, and surfaces of Picard rank one with $K_X$ very ample \cite{BP24}. Very recently Cho \cite{Cho26} gave an algorithm for the cohomology of arbitrary divisors on primary Burniat surfaces and deduced that these carry no Ulrich line bundle for any polarization, while a rank two Ulrich bundle exists for $3K_X$. For surfaces of general type the list is still short, and in every result above resting on the Cayley--Bacharach (or Serre) construction the polarization is assumed non-special, i.e.\ $h^1(X,\OO_X(H))=0$. One purpose of this paper is to explain that, for the natural shape of the construction, this hypothesis is \emph{necessary}.

The work grew out of re-reading two papers from 1994 on stable bundles with large $c_2$, recalled in \S\ref{sec:methods}: Li--Qin \cite{LQ94} use \emph{generic} $0$-cycles, genericity supplying both the Cayley--Bacharach property and the vanishings needed for stability, while \cite{An94} places the cycle on an irreducible curve and enforces both by a B\'ezout argument, following Griffiths--Harris \cite{GH78}. Both aim at $c_2\gg0$, so the lengths are bounded \emph{below}; the Ulrich condition is of the opposite nature, the Chern classes and hence $\len(Z)$ being \emph{fixed} and comparatively small by \eqref{eq:chern}. The generic mechanism works exactly when the polarization is non-special, and then recovers Theorem~\ref{thm:main}; one might hope that the special-position mechanism, designed to work \emph{below} the generic threshold, could cover special polarizations. Theorem~\ref{thm:nogo} shows that it cannot, and not for lack of ingenuity in placing the points --- though it does remain the natural tool for the twisted-kernel variant of \S\ref{subsec:twisted}, where the Cayley--Bacharach system is no longer $|H|$.

\subsection{Main results}\label{subsec:mainresults}
Throughout, $X$ is a smooth projective complex surface. Everything is organized around one existence statement and one obstruction to it. In the form given here the existence statement is due to Casnati \cite{Cas17}; the proof we give, via \cite{LQ94}, is short and elementary (\S\ref{subsec:existence}).

\begin{theorem}[Casnati {\cite[Thm.~1.1]{Cas17}}]\label{thm:main}
Let $X$ be a surface with $p_g(X)=q(X)=0$ and let $H$ be a very ample divisor on $X$ with $h^1(X,\OO_X(H))=0$. Then $X$ carries a special Ulrich bundle $\EE$ of rank two, i.e.\ an Ulrich bundle with
\[
c_1(\EE)=3H+K_X,\qquad c_2(\EE)=\tfrac12\bigl(5H^2+3H\cdot K_X\bigr)+2 .
\]
\end{theorem}

Such an $\EE$ is automatically $\mu_H$-semistable, and it is $\mu_H$-stable unless $X$ carries an Ulrich line bundle, hence in particular whenever $\tfrac12(3H^2+H\cdot K_X)\notin\{M\cdot H:M\in\Pic(X)\}$; this is Corollary~\ref{cor:stable} below, and it is not the stability statement of \cite[Thm.~1.2]{Cas17}. The latter is a complete dichotomy for the same bundles: they are never stable, for any $Z$, if $\OO_X(H)$ embeds $X$ as a rational scroll or if $(X,\OO_X(H))\cong(\PP^2,\OO(1))$, and they are stable for a \emph{general} choice of $Z$ in every other case. The two statements are incomparable --- ours quantifies over all admissible $Z$ at the price of a hypothesis on Ulrich line bundles, \cite[Thm.~1.2]{Cas17} is sharper for general $Z$ and needs no such hypothesis.

The bundle arises from an extension
\begin{equation}\label{eq:shape}
0\longrightarrow\OO_X(H+K_X)\longrightarrow\EE\longrightarrow\OO_X(2H)\otimes\II_Z\longrightarrow 0
\end{equation}
with $Z$ a $0$-dimensional subscheme; we call \eqref{eq:shape} the \emph{adjoint-kernel shape}. Beauville \cite{Bea16b} had reached the same conclusion assuming instead that $|H-K_X|$ contain an irreducible curve --- the stronger hypothesis, as observed in \cite[Rem.~3.5]{Cas17}. Indeed, restricting to $C\in|H-K_X|$ and using $\omega_C=\OO_C(H)$ by adjunction, the sequence $0\to\OO_X(K_X)\to\OO_X(H)\to\omega_C\to0$ places $H^1(\OO_X(H))$ between $H^1(\OO_X(K_X))=0$ and the \emph{kernel} of $H^1(\omega_C)\to H^2(\OO_X(K_X))$; that map is surjective because $h^2(\OO_X(H))=0$, and both spaces have dimension one, so it is an isomorphism and the kernel vanishes. So the Lazarsfeld--Mukai mechanism, though not of the shape \eqref{eq:shape} and hence formally outside Theorem~\ref{thm:nogo}, never applies to a special polarization either.

\smallskip
\noindent\emph{The no-go theorem.} In the literature non-speciality is invariably introduced as a \emph{hypothesis} \cite{Bea16b,Cas17,Cas22,Cas19}. Our main result is that, for the natural shape of the construction, it is \emph{necessary}.

\begin{theorem}[No-go]\label{thm:nogo}
Let $X$ be a surface with $p_g=0$ and $q\ge0$ arbitrary, and let $H$ be an ample divisor with $h^0(\OO_X(H))\neq0$ --- for instance, $H$ very ample --- such that
\[
\chi\bigl(\OO_X(H)\bigr)\ge q\qquad\text{and}\qquad h^1\bigl(X,\OO_X(H)\bigr)\neq0 .
\]
Then there exists \emph{no} rank two vector bundle $\EE$ with
\begin{gather*}
c_1(\EE)=3H+K_X,\qquad c_2(\EE)=\tfrac12\bigl(5H^2+3H\cdot K_X\bigr)+2\chi(\OO_X),\\
h^0\bigl(\EE(-H)\bigr)=0,
\end{gather*}
sitting in an extension of the adjoint-kernel shape \eqref{eq:shape}, for any $0$-dimensional subscheme $Z\subset X$. By Proposition~\ref{prop:reduction} these three conditions hold for every bundle satisfying \eqref{eq:ulrich-vanishing}; in particular, when $H$ is very ample, no special Ulrich bundle of rank two arises from \eqref{eq:shape}.
\end{theorem}

The proof (\S\ref{subsec:nogo}) exhibits a tension between the Cayley--Bacharach property and the vanishing $h^0(\EE(-H))=0$: the Ulrich condition fixes $\len(Z)=\chi(\OO_X(H))+\chi(\OO_X)$ and forces $Z$ to impose conditions on $|H|$ that are independent up to the irregularity, while local freeness forces $Z$ to impose \emph{dependent} conditions --- compatible only when $h^1(\OO_X(H))=0$. Of the six Ulrich vanishings only $h^0(\EE(-H))=0$ enters, together with the value of $c_2$ imposed by the two Riemann--Roch identities; by Proposition~\ref{prop:reduction} this is no loss of information, but it is what the proof needs, and it is worth recording as the hypothesis --- as is the fact that the shape \eqref{eq:shape} may be relaxed to an inclusion of $\OO_X(H+K_X)$ with torsion-free quotient (Remark~\ref{rem:nogo-more}). The auxiliary hypothesis $\chi(\OO_X(H))\ge q$ only guarantees $Z\neq\emptyset$, and when it fails the construction is excluded outright for $q\le3$ (Remark~\ref{rem:nogo-vacuous}). The closest statement we know of is \cite[Rem.~3.4]{Cas17}, where the adjacent converse question --- whether every rank two special Ulrich bundle arises from \eqref{eq:shape} --- is raised and left open; the necessity of non-speciality for the construction itself does not seem to have been recorded. A reader interested only in the regular case may set $q=0$ throughout.

The theorem is about the kernel $\OO_X(H+K_X)$ on the nose, but for $q=0$ its scope is larger: numerical twists and non-saturated adjoint kernels are excluded as well, both by exhibiting a section rather than by counting (Proposition~\ref{prop:numkernel}, Corollary~\ref{cor:numerical-nogo}), which matters on Burniat surfaces, where $\operatorname{Tors}\Pic(X)\cong(\ZZ/2)^6$. For $q>0$ the twisted kernels $\OO_X(H+K_X+\eta)$ used in \cite{Cas19} are genuinely not covered (Remark~\ref{rem:twisted-irregular}), and for $p_g>0$ the adjoint kernel fails at the outset (\S\ref{subsec:pg}).

The hypothesis is satisfied: special very ample divisors on surfaces with $p_g=q=0$ are classical, the \emph{special rational surfaces} of Alexander in $\PP^4$ \cite{Al88,Al92} providing them, the smallest being Okonek's surface of degree $8$ and sectional genus $6$ \cite{Ok86}, whose point configurations were determined by Catanese and Hulek \cite{CH97} and for which $h^1(\OO_S(H))=1$ (Example~\ref{ex:special-rational}). Two consequences quantify how restrictive speciality is (\S\ref{subsec:cost}): a \emph{genus obstruction} (Corollary~\ref{cor:genus}) showing that the mechanism of \cite{An94}, fed with the Ulrich numerology, secretly implies $h^1(\OO_X(H))=0$ and so can only certify polarizations that were non-special all along; and a \emph{restriction identity} (Lemma~\ref{lem:restriction}) identifying the speciality of a polarization with that of a general hyperplane section and bounding it by $\tfrac12H\cdot K_X+1$, so that a special polarization forces $H\cdot K_X\ge0$. On del Pezzo, rational elliptic, Enriques and bielliptic surfaces no polarization is special at all, but the reason is Kawamata--Viehweg rather than the sign of $H\cdot K_X$, which vanishes identically on the last two (Remark~\ref{rem:restriction}(ii)).

\smallskip
\noindent\emph{Higher rank.} As in \cite{CHGS12,Cas17,CasWild}, the rank two bundles generate large families in higher rank. That the surfaces of Theorem~\ref{thm:main} are Ulrich wild is already known, and under exactly the hypothesis of part~(b) below: \cite[Lem.~5.2]{Cas17} states that such a surface with $\pi(H)\ge1$ and $H^2+1\ge K_X^2$ is Ulrich wild, and proves it from the same $\Ext$ estimate fed to the same criterion of Faenzi--Pons-Llopis \cite{FPL}. In the sharper form \cite[Thm.~1.3]{Cas17} the condition on $H^2-K_X^2$ disappears altogether: such a surface is Ulrich wild if and only if either $\pi(H)\ge1$, or $\pi(H)=0$ and $H^2\ge5$, where $\pi(H)=\tfrac12H\cdot(H+K_X)+1$. What the Cayley--Bacharach bundles add is therefore not wildness but finer information --- \emph{simplicity} rather than indecomposability, control of the rank, an explicit dimension count, and independence from \cite[Thm.~1.2]{Cas17} --- deduced from the criterion of \cite{FPL} once the bundles of Theorem~\ref{thm:main} are known to be generically pairwise non-isomorphic (Lemma~\ref{lem:noniso}).

\begin{theorem}\label{thm:evenrank}
In the situation of Theorem~\ref{thm:main}, assume moreover that $X$ is minimal of non-negative Kodaira dimension. Then:
\begin{enumerate}[label=\textup{(\alph*)}]
\item if $H^2+4>K_X^2$, then $X$ carries simple Ulrich bundles of rank $4$;
\item if $H^2+4-K_X^2\ge 3$ \textup{(}automatic, e.g., if $X$ is minimal of general type with $p_g=q=0$ and $H^2\ge 8$, since then $K_X^2\le 9$\textup{)}, then $X$ is strictly Ulrich wild: for every even $r\ge2$ it carries simple Ulrich bundles of rank $r$, moving in families of arbitrarily large dimension as $r$ grows.
\end{enumerate}
\end{theorem}

Only even ranks occur, for a structural reason: the construction feeds an orthogonal \emph{pair} of rank two bundles to the Kronecker quiver functor of \cite{FPL}, which produces bundles of rank $2(a+b)$ (Remark~\ref{rem:evenonly}).

\smallskip
\noindent\emph{An application: primary Burniat surfaces.} Theorem~\ref{thm:nogo} raises the question of where special polarizations actually live. On the most symmetric surfaces of general type with $p_g=q=0$ we answer it completely, and the answer is that they do not live there at all. A \emph{primary Burniat surface} is a minimal surface of general type with $p_g=q=0$ and $K_X^2=6$, a $(\ZZ/2)^2$-cover of the del Pezzo surface $Y$ of degree six branched over the nine Burniat lines and the three exceptional curves (\S\ref{subsec:burniat}). Recall the convention of \S\ref{subsec:ulrich}, that a \emph{polarization} is a \emph{very ample} divisor: part (a) is sharp as stated and false for ample divisors, since $X$ does carry special ample classes (part (b)).

\needspace{6\baselineskip}
\begin{theorem}\label{thm:burniat}
Let $X$ be a primary Burniat surface.
\begin{enumerate}[label=\textup{(\alph*)}]
\item \emph{Every} ample and base point free divisor on $X$ is non-special; in particular so is every polarization. Hence $(X,H)$ carries a $\mu_H$-stable special Ulrich bundle of rank two for every polarization $H$, and $X$ is strictly Ulrich wild with respect to \emph{every} polarization: no additional numerical hypothesis on $H^2$ is needed, since every very ample divisor on $X$ satisfies $H^2\ge18$.
\item All special ample classes are confined to three numerical rays through $K_X$: if $H$ is ample and $h^1(\OO_X(H))\neq0$, then
\[
H\;\equiv\;K_X+c\,(h-f_i)
\]
for some $i\in\{1,2,3\}$ and some integer $c\ge0$, the vertex $c=0$ included. The rays do carry special classes: at their common vertex, three of the $63$ non-trivial torsion twists $K_X+\tau$ are special, and on each ray the classes $H_u$ with $c=2u$, $u\ge2$, are special, with $h^\bullet(\OO_X(H_u))=(3u,u-1,0)$. Not every class on a ray is special. No class on the rays is base point free.
\item For $u\ge1$, $H_u=K_X+\pi^*\OO_Y(u(e_0-e_1))$ and any of the four ramification curves $\varepsilon\in\{\varepsilon_2,\varepsilon_3,\sigma_{12},\sigma_{13}\}$ of $H_u$-degree one, one has $h^\bullet(\OO_X(H_u+\varepsilon))=(3u+1,u,0)$; for $u\ge2$ no twisted-kernel extension built from these data satisfies the Ulrich vanishings.
\end{enumerate}
\end{theorem}

\smallskip
\noindent\emph{Relation with the work of Cho.} The Burniat half of this paper overlaps with \cite{Cho26}, and the overlap deserves a precise statement here rather than in a footnote. The geometric input of part~(a) is a degree bound: an ample and base point free divisor on $X$ has degree at least two on each of the six elliptic ramification curves. That bound is \cite[Prop.~5.4]{Cho26}, obtained there independently and by the same mechanism; we derive it in three lines from an elementary statement about base point free divisors which is not particular to Burniat surfaces: \emph{a base point free divisor on a smooth projective surface has intersection number different from one with every smooth irreducible curve of genus at least one} (Lemma~\ref{lem:bpf-degree}). The six ramification curves are elliptic, and the bound follows. What is not in \cite{Cho26} is the step taken from it: since $K_X$ has degree one on each of the six and the six span $\NE(X)$ (Lemma~\ref{lem:mori}), the bound forces $H-K_X$ to be ample, and Kodaira vanishing then gives $h^1(\OO_X(H))=0$ for \emph{every} polarization. In \cite{Cho26} non-speciality is used only pointwise, for $H=3K_X$, in order to invoke \cite{Cas17}; the global form --- so that the construction runs on $X$ unconditionally, with the wildness consequences of \S\ref{sec:evenrank} --- is not formulated there.

Part~(b) overlaps as well, and unevenly. The passage from $h^1(\OO_X(H))\neq0$ to $(H-K_X)^2=0$ follows in two lines from \cite[Prop.~4.2]{Cho26} and the vanishing step of \cite[Algorithm~4.18(3.1)]{Cho26}; the identification of the ample classes with $(H-K_X)^2=0$ as the three rays is \cite[Prop.~4.10]{Cho26}, there under the additional hypothesis that all $64$ torsion twists in the numerical class be effective. Our route descends to $Y$ and needs no effectivity, so it covers the vertex $c=0$ --- where that hypothesis fails, $e([K_X])$ being $63$ --- and the vertex is where the cheapest special ample classes live. The cohomology along the rays is given in greater generality by \cite[Prop.~4.27]{Cho26}, for every $c$ and every twist, whereas Proposition~\ref{prop:ray} computes the untwisted subfamily with $c$ even; we keep that computation because it is independent of the algorithm and because the same technique yields the ramification-twisted groups of part~(c). The failure of base point freeness on the rays is again a special case of the same degree bound. The explicit ramification-twisted formula of part~(c), and the twisted-kernel obstruction deduced from it, are not stated in \cite{Cho26}, although the cohomology groups involved are in principle accessible through the algorithm there; the whole of \S\ref{sec:ulrich} has no counterpart in \cite{Cho26}.

\smallskip
\noindent\emph{Asymptotics.} Coskun and Huizenga proved, from their asymptotic Brill--Noether theorem, that every smooth projective surface with an ample $H$ carries rank two Ulrich bundles for $mH$, $m\gg0$ \cite[Thm.~1.2]{CH20}; their proof is far from elementary. Section~\ref{sec:asymptotic} examines what \cite{LQ94,An94} yield here: for $p_g=q=0$, the conclusion for every $m$ with $mH-K_X$ nef and big (Proposition~\ref{prop:asymptotic}), the statement in this range being due to Beauville \cite{Bea16b} and the contribution of the elementary proof being the explicit threshold; and beyond it a budget identity (Proposition~\ref{prop:budget}) showing that the counting obstructions close exactly at the boundary whenever the auxiliary curve is rigid, so that what is missing is not a numerical but a positional statement about $0$-cycles (Question~\ref{q:cycles}).

\subsection*{Conventions}
We work over $\CC$. We write $H^i(\FF)$ for $H^i(X,\FF)$ and $h^i$ for its dimension; $K=K_X$ is the canonical divisor, $p_g=h^0(K)$, $q=h^1(\OO_X)$, $\chi=\chi(\OO_X)$. For a $0$-dimensional $Z\subset X$, $\len(Z)$ is its length and $\II_Z$ its ideal sheaf. Stability is slope stability in the sense of Mumford--Takemoto, as in \cite{LQ94}. Following \S\ref{subsec:ulrich}, a \emph{polarization} is always a \emph{very ample} divisor; classes that are ample but not base point free, such as the $H_u$ of Proposition~\ref{prop:ray}, are called ample classes and never polarizations. \emph{Base point free} and \emph{globally generated} are used interchangeably. Finally we record a triviality used repeatedly: \emph{if $p_g=0$ and $D$ is a non-zero effective divisor, then $h^0(\OO_X(K_X-D))=0$}, since multiplication by a non-zero section of $\OO_X(D)$ embeds $H^0(\OO_X(K_X-D))$ into $H^0(\OO_X(K_X))=0$; in particular $h^2(\OO_X(D))=0$ and $h^0(\OO_X(D))=\chi(\OO_X(D))+h^1(\OO_X(D))$.

\section{The two constructions}\label{sec:methods}

This section collects the two mechanisms on which everything below rests: the genericity package of Li and Qin \cite{LQ94}, recalled in \S\ref{subsec:LQ}, and the points-on-a-curve device of \cite{An94}, formalized in \S\ref{subsec:Anghel}. Both feed the extension criterion of \S\ref{subsec:CB}, which we record first.

\subsection{The Cayley--Bacharach property and extensions}\label{subsec:CB}
The starting point of both \cite{LQ94} and \cite{An94} is the following classical criterion (see \cite[p.~731]{GH78book}, \cite[Prop.~2.2]{LQ94}).

\begin{definition}\label{def:CB}
Let $|D|$ be a (possibly empty) complete linear system on $X$ and $Z\subset X$ a $0$-dimensional subscheme. We say that $Z$ satisfies the \emph{Cayley--Bacharach property with respect to $|D|$}, written CB$(|D|)$, if for every subscheme $Z'\subseteq Z$ of colength one,
\[
H^0\bigl(\OO_X(D)\otimes\II_{Z'}\bigr)=H^0\bigl(\OO_X(D)\otimes\II_{Z}\bigr).
\]
For $Z$ reduced this is the classical statement: every curve in $|D|$ containing all but one point of $Z$ contains the remaining point.
\end{definition}

\begin{proposition}[Serre, Griffiths--Harris]\label{prop:CB}
Let $L',L''\in\Pic(X)$ and let $Z\subset X$ be a $0$-dimensional \emph{locally complete intersection} subscheme. There exists an extension
\begin{equation}\label{eq:ext}
0\longrightarrow \OO_X(L')\longrightarrow \EE\longrightarrow \OO_X(L'')\otimes\II_Z\longrightarrow 0
\end{equation}
with $\EE$ \emph{locally free} if and only if $Z$ satisfies \textup{CB}$(|L''-L'+K_X|)$. In particular this holds if
\begin{equation}\label{eq:strongCB}
h^0\bigl(\OO_X(L''-L'+K_X)\otimes\II_{Z'}\bigr)=0\qquad\text{for every colength-one }Z'\subseteq Z,
\end{equation}
cf.\ \cite[Cor.~2.3]{LQ94}; see \cite{Arr07} for a detailed account of the Hartshorne--Serre correspondence.
\end{proposition}

Two remarks on the hypotheses. First, every \emph{reduced} $0$-dimensional subscheme is locally complete intersection, so the hypothesis is harmless for the reduced cycles used throughout this note. Second, the l.c.i.\ condition is needed only in the \emph{existence} direction: if an extension \eqref{eq:ext} with $\EE$ locally free exists for some $0$-dimensional subscheme $Z$, then locally $\II_Z$ admits a length-one resolution $0\to\OO\to\OO^{2}\to\II_Z\to0$, so $\II_Z$ is locally generated by two elements --- this is the Hilbert--Burch theorem in its simplest case \cite[Thm.~20.15]{Eis95} --- and, an ideal of height two generated by two elements in a two-dimensional regular local ring being generated by a regular sequence, $Z$ is \emph{automatically} locally complete intersection. The hypothesis of the equivalence above is therefore met a posteriori, and its ``only if'' direction applies: $Z$ satisfies CB$(|L''-L'+K_X|)$. Thus the necessary direction, which is all that the no-go Theorem~\ref{thm:nogo} uses, holds for arbitrary $0$-dimensional $Z$, the l.c.i.\ hypothesis being a consequence rather than an assumption.

For a bundle $\EE$ as in \eqref{eq:ext} one has $c_1(\EE)=L'+L''$ and $c_2(\EE)=L'\cdot L''+\len(Z)$.

\subsection{The generic mechanism of Li--Qin}\label{subsec:LQ}
The engine of \cite{LQ94} consists of two elementary genericity lemmas, which we quote in the form we need (they are Lemmas 2.4 and 2.5 of \cite{LQ94}).

\begin{lemma}[{\cite[Lem.~2.4]{LQ94}}]\label{lem:LQ24}
Let $M$ be a divisor on $X$ and let $Z\in\Hilb^{\len}(X)$ be generic. If $\len\ge h^0(\OO_X(M))$ then $h^0(\OO_X(M)\otimes\II_Z)=0$.
\end{lemma}

\begin{lemma}[{\cite[Lem.~2.5]{LQ94}}]\label{lem:LQ25}
Let $M$ be a divisor on $X$ and let $\len$ be an integer with $\len\ge p_g$ and $\len\ge h^0(\OO_X(M))$. Then a generic reduced $Z\in\Hilb^{\len+1}(X)$ satisfies \eqref{eq:strongCB} for the system $|M|$; in particular $Z$ satisfies \textup{CB}$(|M|)$.
\end{lemma}

Thus for generic $0$-cycles the Cayley--Bacharach property is available exactly at the threshold $\len(Z)\ge h^0(\OO_X(M))+1$ (when $p_g=0$). The third lemma of \cite{LQ94}, which yields stability, is not needed here: for Ulrich bundles stability comes for free from semistability (Lemma~\ref{lem:semistable}).

\subsection{The points-on-a-curve mechanism}\label{subsec:Anghel}
The key idea of \cite{An94} is to abandon genericity and place the $0$-cycle on a single irreducible curve; the required vanishings then follow from B\'ezout. We formalize the mechanism as follows.

\begin{lemma}\label{lem:bezout}
Let $|M|$ be a linear system on $X$ and let $C\subset X$ be a reduced irreducible curve with
\[
h^0\bigl(\OO_X(M-C)\bigr)=0 .
\]
Let $Z$ be a set of $\len$ distinct points on $C$ with $\len\ge M\cdot C+2$. Then
\[
h^0\bigl(\OO_X(M)\otimes\II_{Z\smallsetminus\{x\}}\bigr)=0\quad\text{for every }x\in Z;
\]
in particular $h^0(\OO_X(M)\otimes\II_Z)=0$ and $Z$ satisfies \textup{CB}$(|M|)$.
\end{lemma}

\begin{proof}
Suppose $D\in|M|$ contains $Z\smallsetminus\{x\}$, a set of $\len-1\ge M\cdot C+1$ distinct points of $C$. If $C\not\subset D$, then $D\cap C$ is finite and
\[
M\cdot C=D\cdot C\;\ge\;\#(D\cap C)\;\ge\;\len-1\;\ge\;M\cdot C+1,
\]
a contradiction. Hence $C\subset D$ and $D-C$ is an effective divisor in $|M-C|$, contradicting $h^0(\OO_X(M-C))=0$. Thus no such $D$ exists, which is statement \eqref{eq:strongCB}; the Cayley--Bacharach property follows trivially, both sides of the defining equality being zero.
\end{proof}

\begin{remark}\label{rem:anghel-role}
This is the argument of \cite[Th\'eor\`eme~2.2]{An94}, with $C=H_1\in|n_LH|$ and $|M|$ the adjoint system. Since it imposes no lower bound coming from $h^0(|M|)$, one might hope to run the Ulrich construction \eqref{eq:shape} below the generic threshold of Lemma~\ref{lem:LQ25}, i.e.,\ for \emph{special} polarizations; Theorem~\ref{thm:nogo} shows that this is impossible for the adjoint kernel, the hypotheses of Lemma~\ref{lem:bezout} fed with the Ulrich numerology secretly implying $h^1(\OO_X(H))=0$ (Corollary~\ref{cor:genus}). The lemma regains its force for the twisted kernels of \S\ref{subsec:twisted}.
\end{remark}

\section{The Cayley--Bacharach construction and the no-go theorem}\label{sec:ulrich}

\subsection{A numerical reduction}
We first record how the Ulrich condition interacts with extensions of the form \eqref{eq:ext}.

\begin{lemma}\label{lem:consistency}
Let $\EE$ be a rank $r$ bundle given by an extension
$0\to\OO_X(A)\to\EE\to\bigoplus_{i=1}^{r-1}\OO_X(B_i)\otimes\II_{Z_i}\to 0$. Then
\[
\chi\bigl(\EE(-H)\bigr)-\chi\bigl(\EE(-2H)\bigr)
= c_1(\EE)\cdot H-\frac{r}{2}\bigl(3H^2+H\cdot K_X\bigr).
\]
In particular the two conditions $\chi(\EE(-H))=\chi(\EE(-2H))=0$ are compatible exactly when $c_1(\EE)$ satisfies the Ulrich degree condition in \eqref{eq:chern}, and then they impose a single condition on $\sum_i\len(Z_i)$.
\end{lemma}

\begin{proof}
For any divisor $M$, Riemann--Roch gives
$\chi(\OO_X(M-H))-\chi(\OO_X(M-2H))=M\cdot H-\tfrac12(3H^2+H\cdot K_X)$,
and the lengths $\len(Z_i)$ cancel in the difference. Summing over the sub- and quotient line bundles yields the claim.
\end{proof}

\begin{remark}\label{rem:rank2forced}
If one insists on the ``uniform'' choice $A=H+K_X$, $B_i=2H$ (the one that makes all $h^0$-vanishings come for free when $p_g=q=0$, see below), the degree condition of Lemma~\ref{lem:consistency} reads $(r-2)\,H\cdot(H-K_X)=0$. Away from the degenerate case $H^2=H\cdot K_X$ this \emph{forces $r=2$}: the line-bundle-kernel Cayley--Bacharach construction naturally produces rank two special Ulrich bundles, and higher rank must be reached either by mixing the twists $B_i$ or by iterated extensions of rank two Ulrich bundles (\S\ref{sec:evenrank}).
\end{remark}

The rank two special case reduces to a single section-vanishing:

\begin{proposition}[cf.\ {\cite[Prop.~2.1(4) and Cor.~2.4]{Cas17}}]\label{prop:reduction}
Let $X$ be a surface and let $H$ be a divisor with $h^0(\OO_X(H))\neq0$ --- for instance, $H$ very ample. Let $\EE$ be a rank two vector bundle with $c_1(\EE)=3H+K_X$ and $c_2(\EE)=\tfrac12(5H^2+3H\cdot K_X)+2\chi(\OO_X)$. Then $\EE$ is a (special) Ulrich bundle if and only if
\[
h^0\bigl(\EE(-H)\bigr)=0 .
\]
For $H$ very ample this is \cite[Cor.~2.4]{Cas17}, where the criterion is stated as: a rank two bundle with the special Chern classes is Ulrich if and only if it is initialized.
\end{proposition}

\begin{proof}
Necessity is clear. Conversely, multiplication by a non-zero section of $\OO_X(H)$ embeds $\EE(-2H)$ into $\EE(-H)$, so that $h^0(\EE(-2H))\le h^0(\EE(-H))=0$: the second vanishing is free. Since $\rk\EE=2$ we have $\EE^\vee\cong\EE(-c_1)=\EE(-3H-K_X)$, so Serre duality gives
\begin{gather*}
h^2\bigl(\EE(-H)\bigr)=h^0\bigl(\EE^\vee(H+K_X)\bigr)=h^0\bigl(\EE(-2H)\bigr)=0,\\
h^2\bigl(\EE(-2H)\bigr)=h^0\bigl(\EE^\vee(2H+K_X)\bigr)=h^0\bigl(\EE(-H)\bigr)=0 .
\end{gather*}
By Riemann--Roch, $\chi(\EE(-H))=2\chi(\OO_X)+\tfrac12(5H^2+3H\cdot K_X)-c_2(\EE)=0$, and $\chi(\EE(-2H))=\chi(\EE(-H))=0$ by the same duality. Hence $h^1(\EE(-H))=h^1(\EE(-2H))=0$ as well, and $\EE$ is Ulrich.
\end{proof}

\subsection{Existence: proof of Theorem \ref{thm:main}}\label{subsec:existence}
Assume $p_g=q=0$, so $\chi(\OO_X)=1$. Set
\[
\len\;:=\;\tfrac12\,H\cdot(H-K_X)+2\;=\;\chi\bigl(\OO_X(H)\bigr)+1 .
\]
(The second equality is Riemann--Roch with $\chi(\OO_X)=1$.) We construct $\EE$ as an extension of the adjoint-kernel shape \eqref{eq:shape} with $\len(Z)=\len$. Note that $c_1(\EE)=3H+K_X$ and
$c_2(\EE)=(H+K_X)\cdot 2H+\len=\tfrac12(5H^2+3H\cdot K_X)+2$,
as required by \eqref{eq:c2special}. The relevant Cayley--Bacharach system in Proposition~\ref{prop:CB} is
\[
\bigl|\,2H-(H+K_X)+K_X\,\bigr|=|H| .
\]

\begin{proof}[Proof of Theorem \textup{\ref{thm:main}}]
\emph{Step 1: the cycle $Z$.} Since $h^1(\OO_X(H))=h^2(\OO_X(H))=0$ (Conventions), $h^0(\OO_X(H))=\chi(\OO_X(H))=\len-1$. Since $p_g=0$, we may apply Lemma~\ref{lem:LQ25} with $|M|=|H|$ and with its length parameter set to $\len-1$ (admissible, as $\len-1=h^0(\OO_X(H))\ge p_g$) shows that a generic reduced $Z\in\Hilb^{\len}(X)$ satisfies \eqref{eq:strongCB} for $|H|$; in particular $h^0(\OO_X(H)\otimes\II_Z)=0$ and $Z$ satisfies CB$(|H|)$.

\smallskip
\noindent \emph{Step 2: the bundle.} Since $Z$ is reduced, hence locally complete intersection, Proposition~\ref{prop:CB} provides a locally free extension \eqref{eq:shape}. (For the record, the space of such extensions is
$\Ext^1(\OO_X(2H)\otimes\II_Z,\OO_X(H+K_X))\cong H^1(\OO_X(H)\otimes\II_Z)^\vee$,
by Serre duality for Ext groups; from $0\to\OO_X(H)\otimes\II_Z\to\OO_X(H)\to\OO_Z\to 0$, Step 1 and $\len=\chi(\OO_X(H))+1$ one computes $h^1(\OO_X(H)\otimes\II_Z)=\len-\chi(\OO_X(H))=1$, so the bundle $\EE$ is in fact \emph{uniquely determined} by $Z$; cf.\ \cite[Rem.~3.4]{Cas17}.)

\smallskip
\noindent \emph{Step 3: Ulrich vanishings.} Twisting \eqref{eq:shape} by $-H$ and $-2H$ gives
\begin{gather*}
0\to\OO_X(K_X)\to\EE(-H)\to\OO_X(H)\otimes\II_Z\to 0,\\
0\to\OO_X(K_X-H)\to\EE(-2H)\to\II_Z\to 0 .
\end{gather*}
Since $p_g=0$ and $h^0(\OO_X(H)\otimes\II_Z)=0$ (Step 1), we get $h^0(\EE(-H))=0$. Moreover $h^0(\OO_X(K_X-H))\le h^0(\OO_X(K_X))=0$ and $h^0(\II_Z)=0$ because $Z\neq\emptyset$ ($\len\ge 2$), whence $h^0(\EE(-2H))=0$. By Proposition~\ref{prop:reduction}, $\EE$ is a special Ulrich bundle.

\smallskip
\noindent \emph{Step 4: (semi)stability.} This is Lemma~\ref{lem:semistable} and Corollary~\ref{cor:stable} below.
\end{proof}

\begin{lemma}[{\cite[Thm.~2.9]{CHGS12}}; cf.\ {\cite[Thm.~4.1(3)]{Cas17}}]\label{lem:semistable}
Every Ulrich bundle is $\mu_H$-semistable; and if $0\to\mathcal L\to\FF\to\mathcal M\to0$ is exact with $\FF$ Ulrich, $\mathcal M$ torsion free and $\mu_H(\mathcal L)=\mu_H(\FF)$, then $\mathcal L$ and $\mathcal M$ are Ulrich as well. Consequently, a strictly $\mu_H$-semistable rank two Ulrich bundle is an extension of Ulrich line bundles: a destabilizing subsheaf may be replaced by its saturation, whose quotient is then torsion free.
\end{lemma}

\begin{corollary}\label{cor:stable}
The bundle $\EE$ of Theorem~\ref{thm:main} is $\mu_H$-stable unless $X$ carries an Ulrich line bundle $M$; any such $M$ satisfies $M\cdot H=\tfrac12(3H^2+H\cdot K_X)$. In particular, if $\tfrac12(3H^2+H\cdot K_X)$ is not the $H$-degree of any line bundle on $X$, then $\EE$ is $\mu_H$-stable.
\end{corollary}

\begin{proof}
By Lemma~\ref{lem:semistable}, $\EE$ is $\mu_H$-semistable, and if it fails to be $\mu_H$-stable it is an extension of Ulrich line bundles. An Ulrich line bundle $M$ has rank one, so \eqref{eq:chern} gives $M\cdot H=\tfrac12(3H^2+H\cdot K_X)$. The last assertion is immediate.
\end{proof}

\subsection{The no-go result}\label{subsec:nogo}
We now prove Theorem~\ref{thm:nogo}. The point is that the two requirements of the construction collide for the adjoint kernel: local freeness, i.e.,\ the Cayley--Bacharach property, makes $Z$ impose \emph{dependent} conditions on $|H|$, while the Ulrich vanishing makes $Z$ impose \emph{independent} conditions --- and the two are compatible only when $h^1(\OO_X(H))=0$.

\begin{proof}[Proof of Theorem \textup{\ref{thm:nogo}}]
Suppose $\EE$ satisfies the hypotheses of the theorem, that is, $\EE$ has rank two, $c_1(\EE)=3H+K_X$, $c_2(\EE)$ is as in \eqref{eq:c2special}, $h^0(\EE(-H))=0$, and $\EE$ sits in \eqref{eq:shape}. Only $H$ ample with $h^0(\OO_X(H))\neq0$ is used; very ampleness plays no role. We derive $h^1(\OO_X(H))=0$ in four steps.

\emph{(1)} From $c_2(\EE)=(H+K_X)\cdot 2H+\len(Z)$ and \eqref{eq:c2special}:
\[
\len(Z)=\tfrac12H\cdot(H-K_X)+2\chi(\OO_X)=\chi\bigl(\OO_X(H)\bigr)+\chi(\OO_X) .
\]
Since $p_g=0$ we have $\chi(\OO_X)=1-q$, so $\len(Z)=\chi(\OO_X(H))+1-q\ge1$ by hypothesis, and in particular $Z\neq\emptyset$.

\emph{(2)} Twist \eqref{eq:shape} by $-H$ and take cohomology:
\[
0\to H^0(\OO_X(K_X))\to H^0(\EE(-H))\to H^0(\OO_X(H)\otimes\II_Z)\to H^1(\OO_X(K_X)).
\]
Since $p_g=0$ we have $H^0(\OO_X(K_X))=0$, while $H^1(\OO_X(K_X))\cong H^1(\OO_X)^\vee$ has dimension $q$. The Ulrich condition $h^0(\EE(-H))=0$ therefore makes the connecting map injective on $H^0(\OO_X(H)\otimes\II_Z)$, whence
\[
h^0\bigl(\OO_X(H)\otimes\II_Z\bigr)\;\le\;q .
\]
(For $q=0$ this is the vanishing $h^0(\OO_X(H)\otimes\II_Z)=0$.)

\emph{(3)} Since $\EE$ is locally free, the necessary direction of Proposition~\ref{prop:CB} --- which, as noted after its statement, holds for \emph{arbitrary} $0$-dimensional $Z$ --- says $Z$ satisfies CB$(|H|)$: for every colength-one $Z'\subseteq Z$,
\[
h^0\bigl(\OO_X(H)\otimes\II_{Z'}\bigr)=h^0\bigl(\OO_X(H)\otimes\II_{Z}\bigr)\;\le\;q .
\]

\emph{(4)} On the other hand, a subscheme of length $\len(Z)-1$ imposes at most $\len(Z)-1$ conditions on $|H|$. Using $h^0(\OO_X(H))=\chi(\OO_X(H))+h^1(\OO_X(H))$ (Conventions) together with $\len(Z)-1=\chi(\OO_X(H))-q$ from step (1), we get
\[
h^0\bigl(\OO_X(H)\otimes\II_{Z'}\bigr)\ge h^0\bigl(\OO_X(H)\bigr)-\bigl(\len(Z)-1\bigr)=h^1\bigl(\OO_X(H)\bigr)+q .
\]
Combining with (3) gives $h^1(\OO_X(H))+q\le q$, i.e.\ $h^1(\OO_X(H))\le0$, as claimed.
\end{proof}

\begin{remark}\label{rem:nogo-more}
The proof shows more. On a surface with $p_g=0$, and assuming as above that $\chi(\OO_X(H))\ge q$, \emph{every} special Ulrich bundle of rank two containing $\OO_X(H+K_X)$ as a subsheaf with torsion-free quotient certifies $h^1(\OO_X(H))=0$, whatever $q$ --- the budget deficit and the absorber $H^1(\OO_X(K_X))$ grow by the same $q$ and cancel. The non-speciality hypothesis in \cite{Bea16b,Cas17,Cas22} is thus forced by the shape of the construction, not by the method of choosing $Z$. Note where the obstruction lives: the connecting homomorphism in step (2) lands in $H^1(\OO_X(K_X))\cong H^1(\OO_X)^\vee$, of dimension exactly $q$ --- so the absorber is trivial precisely in the regular case, and it is the adjointness of the kernel that pins the target down to this space in the first place. This is the seed of \S\ref{subsec:twisted}. The hypothesis $\chi(\OO_X(H))\ge q$ is needed here for the same reason as in the theorem, namely so that step~(4) has a colength-one subscheme at its disposal: when $\len(Z)=0$ the argument still gives $h^0(\OO_X(H))\le q$, hence $h^1(\OO_X(H))\le1$, but no longer $h^1(\OO_X(H))=0$.
\end{remark}

\begin{remark}[When the auxiliary hypothesis fails]\label{rem:nogo-vacuous}
The role of $\chi(\OO_X(H))\ge q$ is only to guarantee that $Z$ be non-empty. If it fails, step~(1) gives $\len(Z)\le0$: when $\chi(\OO_X(H))<q-1$ no bundle of the shape \eqref{eq:shape} exists at all, the length being negative, and when $\chi(\OO_X(H))=q-1$ one has $Z=\emptyset$ while step~(2) still yields $h^0(\OO_X(H))\le q$. Suppose in addition that $H$ is very ample. Then $|H|$ embeds $X$ in $\PP^{h^0(\OO_X(H))-1}$ as a non-degenerate surface, and a surface spanning $\PP^2$ is $\PP^2$ itself; so $h^0(\OO_X(H))\ge4$ unless $(X,\OO_X(H))\cong(\PP^2,\OO(1))$, and that exception has $q=0$, where the hypothesis holds automatically. The construction is therefore excluded outright whenever $q\le3$. For $q=0$ no positivity beyond ampleness is needed, the bound $h^0(\OO_X(H))\le q=0$ contradicting the standing hypothesis $h^0(\OO_X(H))\neq0$ on its own; but for $1\le q\le3$ and $H$ merely ample no such conclusion is available, an ample divisor need not be base point free, let alone very ample, and $h^0(\OO_X(H))$ can be as small as $1$.
\end{remark}

\begin{proposition}[Numerical twists and non-saturated kernels]\label{prop:numkernel}
Let $X$ be a surface with $p_g=q=0$, let $H$ be very ample and let $\EE$ be a rank two bundle with $h^0(\EE(-H))=0$ --- in particular, any Ulrich bundle with respect to $H$. Then:
\begin{enumerate}[label=\textup{(\alph*)}]
\item $\EE$ contains no subsheaf $\OO_X(A)$ with $A\equiv H+K_X$ numerically and $A\neq H+K_X$ in $\Pic(X)$ --- no saturation or torsion-freeness of the quotient is required;
\item if $s\in H^0(\EE(-H-K_X))$ vanishes on a divisor $D>0$, then $h^1(\OO_D)=0$; in particular $p_a(D)\le0$ and every irreducible component of $D$ is a smooth rational curve.
\end{enumerate}
\end{proposition}

\begin{proof}
In both cases $\EE(-H)$ contains $\OO_X(A-H)$ as a subsheaf, so it suffices to produce a non-zero section of the latter.

(a) Here $A-H=K_X+\tau$ with $\tau$ numerically trivial and, since $q=0$ forces $\Pic^0(X)=0$, torsion; by hypothesis $\tau\neq0$. Then $h^2(\OO_X(K_X+\tau))=h^0(\OO_X(-\tau))=0$, while $\tau\cdot\tau=\tau\cdot K_X=0$ gives $\chi(\OO_X(K_X+\tau))=\chi(\OO_X)=1$, whence
\[
h^0\bigl(\OO_X(K_X+\tau)\bigr)=1+h^1\bigl(\OO_X(K_X+\tau)\bigr)\;\ge\;1 .
\]

(b) Here $A-H=K_X+D$. By adjunction $(K_X+D)|_D=\omega_D$, so the restriction sequence reads $0\to\OO_X(K_X)\to\OO_X(K_X+D)\to\omega_D\to0$; since $h^0(\OO_X(K_X))=p_g=0$ and $h^1(\OO_X(K_X))=h^1(\OO_X)^\vee=0$, it gives $h^0(\OO_X(K_X+D))=h^0(\omega_D)=h^1(\OO_D)$, the last equality by duality on the Gorenstein curve $D$. So $h^1(\OO_D)=0$. The consequences follow: $p_a(D)=1-\chi(\OO_D)=1-h^0(\OO_D)+h^1(\OO_D)\le0$ since $h^0(\OO_D)\ge1$; and for an irreducible component $C\subseteq D$, taken with its reduced structure, the surjection $\OO_D\twoheadrightarrow\OO_C$ has kernel supported in dimension $\le1$, so $H^1(\OO_D)\twoheadrightarrow H^1(\OO_C)$ and $p_a(C)=h^1(\OO_C)=0$: $C$ is a smooth rational curve.
\end{proof}

\begin{corollary}\label{cor:numerical-nogo}
Let $X$ have $p_g=q=0$ and let $H$ be very ample with $h^1(\OO_X(H))\neq0$. Then no rank two Ulrich bundle for $H$ arises from an extension
\[
0\longrightarrow\OO_X(A)\longrightarrow\EE\longrightarrow\OO_X(B)\otimes\II_Z\longrightarrow0,\qquad A\equiv H+K_X,\quad B\equiv 2H ,
\]
for any $0$-dimensional $Z$ and any such $A,B$ with $A+B=3H+K_X$ in $\Pic(X)$.
\end{corollary}

\begin{proof}
If $A\neq H+K_X$ this is Proposition~\ref{prop:numkernel}(a), and holds with no hypothesis on $h^1(\OO_X(H))$. If $A=H+K_X$ then $B=2H$ and this is Theorem~\ref{thm:nogo}, whose auxiliary hypothesis $\chi(\OO_X(H))\ge q=0$ is satisfied here: by Remark~\ref{rem:nogo-vacuous}, $\chi(\OO_X(H))=-1$ would force $h^0(\OO_X(H))\le0$, contradicting very ampleness, and $\chi(\OO_X(H))<-1$ would force $\len(Z)<0$.
\end{proof}

\begin{remark}\label{rem:numerical-scope}
Two comments on the scope. First, if one weakens the definition of \emph{special} to $c_1(\EE)\equiv3H+K_X$ numerically, the remaining configuration is $A=H+K_X$, $B=2H+\beta$ with $\beta$ torsion; there the vanishing system and the Cayley--Bacharach system $|B-A+K_X|=|H+\beta|$ still coincide, and the proof of Theorem~\ref{thm:nogo} gives $h^1(\OO_X(H+\beta))=0$ --- non-speciality of the twisted class rather than of $H$ itself. One point needs an argument, since step~(4) uses $h^0(\OO_X(H+\beta))=\chi+h^1$, hence $h^2(\OO_X(H+\beta))=0$, and $H+\beta$ need not be effective: if $h^2(\OO_X(H+\beta))\neq0$ then $E:=K_X-H-\beta$ is effective, so $H\cdot K_X\ge H^2$; on the other hand $\EE$ is $\mu_H$-semistable (Lemma~\ref{lem:semistable}) and the subsheaf $\OO_X(H+K_X)$ gives $(H+K_X)\cdot H\le\tfrac12(3H^2+H\cdot K_X)$, i.e.\ $H\cdot K_X\le H^2$. Hence $H\cdot E=0$ with $H$ ample and $E$ effective, so $E=0$ and $H+\beta=K_X$, in which case $h^1(\OO_X(H+\beta))=h^1(\OO_X(K_X))=q=0$ anyway. Second, Proposition~\ref{prop:numkernel}(b) says that the adjoint kernel may fail to be saturated only along divisors all of whose irreducible components are rational curves; combined with Corollary~\ref{cor:genus}, which shows that rational curves can only certify polarizations that were non-special all along, this closes the gap between ``no bundle of the shape \eqref{eq:shape}'' and ``no bundle containing $\OO_X(H+K_X)$'' on any surface carrying no rational curves.
\end{remark}

\begin{remark}\label{rem:twisted-irregular}
For $q>0$ the picture is genuinely different, and Proposition~\ref{prop:numkernel}(a) has no analogue. There $\Pic^0(X)$ is positive-dimensional and, for $\eta\in\Pic^0(X)$ non-zero, $\chi(\OO_X(K_X+\eta))=1-q$, so that $h^0(\OO_X(K_X+\eta))=0$ is the expected behaviour and the absorber $H^1(\OO_X(K_X+\eta))$, of dimension $q-1$, is available. For the kernel $\OO_X(H+K_X+\eta)$ with quotient $\OO_X(2H+\eta)\otimes\II_Z$ the vanishing system is $|H+\eta|$ while the Cayley--Bacharach system is $|H|$: the two estimates of steps (2) and (3)--(4) below concern different linear systems and do not combine. Theorem~\ref{thm:nogo} says nothing about such kernels, which are precisely the shape used in the irregular case in \cite{Cas19}.
\end{remark}

\subsection{The genus obstruction and the cost of speciality}\label{subsec:cost}
The following corollary quantifies the collapse of the points-on-a-curve approach for the adjoint kernel; it is the computation that emerges from testing the mechanism on concrete surfaces (\S\ref{subsec:burniat}).

\begin{corollary}[Genus obstruction]\label{cor:genus}
Let $X$ have $p_g=q=0$, let $H$ be ample and let $C\subset X$ be a reduced irreducible curve with
\begin{equation}\label{eq:star-b}
H\cdot C\;\le\;\tfrac12\,H\cdot(H-K_X)
\end{equation}
\textup{(}the B\'ezout condition for Lemma~\ref{lem:bezout}, run with $|M|=|H|$ and $\len=\chi(\OO_X(H))+1$ points on $C$: the requirement there is $\len\ge H\cdot C+2$, that is, $\tfrac12H\cdot(H-K_X)+2\ge H\cdot C+2$, which is \eqref{eq:star-b}\textup{)}. Then
\[
h^0\bigl(\OO_X(H-C)\bigr)\ge\chi\bigl(\OO_X(H-C)\bigr)=\tfrac12H\cdot(H-K_X)-H\cdot C+p_a(C)\ge p_a(C).
\]
In particular, if $h^0(\OO_X(H-C))=0$ \textup{(}the residual condition needed in Lemma~\ref{lem:bezout}\textup{)}, then necessarily $p_a(C)=0$ --- so $C\cong\PP^1$ --- and moreover $H\cdot C=\tfrac12H\cdot(H-K_X)$ exactly, $h^1(\OO_X(H-C))=0$, and $h^1(\OO_X(H))=0$: the polarization was non-special all along.
\end{corollary}

\begin{proof}
By Riemann--Roch and adjunction ($C^2+C\cdot K_X=2p_a(C)-2$),
\[
\chi\bigl(\OO_X(H-C)\bigr)=1+\tfrac12(H-C)\cdot(H-C-K_X)=\tfrac12H\cdot(H-K_X)-H\cdot C+p_a(C),
\]
which is $\ge p_a(C)$ by \eqref{eq:star-b}. For $h^2$: if $K_X-H+C$ were effective (or zero), intersecting with the ample $H$ would give $H\cdot C\ge H\cdot(H-K_X)$, which together with \eqref{eq:star-b} forces $H\cdot(H-K_X)\le0$, hence $H\cdot C\le 0$, impossible for an irreducible curve. So $h^2(\OO_X(H-C))=h^0(\OO_X(K_X-H+C))=0$ and $h^0\ge\chi$.

If $h^0(\OO_X(H-C))=0$ then $0\ge\chi(\OO_X(H-C))\ge p_a(C)\ge0$, forcing $p_a(C)=0$ (an integral curve with $p_a=0$ is a smooth rational curve), $H\cdot C=\tfrac12H\cdot(H-K_X)$ and $h^1(\OO_X(H-C))=-\chi=0$. Finally, the restriction sequence
$0\to\OO_X(H-C)\to\OO_X(H)\to\OO_{\PP^1}(H\cdot C)\to0$
gives $H^1(\OO_X(H))$ sandwiched between $H^1(\OO_X(H-C))=0$ and $H^1(\OO_{\PP^1}(H\cdot C))=0$.
\end{proof}

\medskip
Corollary~\ref{cor:genus} and Lemma~\ref{lem:restriction} below are two readings of one and the same exact sequence,
\[
0\longrightarrow\OO_X(H-C)\longrightarrow\OO_X(H)\longrightarrow\OO_C(H)\longrightarrow 0 ,
\]
taken at the two ends of the range of admissible $C$. Corollary~\ref{cor:genus} takes $C$ of small $H$-degree --- the curve on which one would like to place the $0$-cycle --- and observes that the residual vanishing $h^0(\OO_X(H-C))=0$ needed to run Lemma~\ref{lem:bezout} already squeezes $H^1(\OO_X(H))$ between two vanishing groups. Lemma~\ref{lem:restriction} takes $C\in|H|$, where the residual term degenerates to $\OO_X$ itself and the sequence computes $h^1(\OO_X(H))$ outright. The point-placement mechanism and the polarization are thus obstructed by the same sequence: in the first case one has bought the vanishing of $h^1$ without noticing, in the second one sees what $h^1$ is.

\medskip
\noindent\emph{What speciality costs.} It is worth recording how restrictive speciality itself is: for a regular surface with $p_g=0$ it is exactly the speciality of the hyperplane series on a general hyperplane section, and Clifford's theorem then bounds it linearly in $H\cdot K_X$. The identity itself is standard, being the restriction sequence read once; in one direction it is used in \cite[Lem.~3.1 and Ex.~5.4]{Cas17}, where non-special surfaces in $\PP^4$ are observed to be sectionally non-special. What we shall use is the resulting bound.

\begin{lemma}\label{lem:restriction}
Let $X$ be a surface with $p_g(X)=q(X)=0$, let $H$ be very ample and let $C\in|H|$ be a general member, a smooth irreducible curve of genus $g=\tfrac12H\cdot(H+K_X)+1$. Then
\[
h^1\bigl(\OO_X(H)\bigr)\;=\;h^1\bigl(C,\OO_C(H)\bigr)\;=\;h^0\bigl(C,\,K_X|_C\bigr).
\]
Moreover, if $h^1(\OO_X(H))\neq0$ then
\[
h^1\bigl(\OO_X(H)\bigr)\;\le\;\tfrac12\,H\cdot K_X+1 ,
\]
with equality only if $C$ is hyperelliptic or $K_X|_C\cong\OO_C$; in particular $H\cdot K_X\ge2\bigl(h^1(\OO_X(H))-1\bigr)\ge0$.
\end{lemma}

\begin{proof}
The general member of $|H|$ is smooth and irreducible by Bertini, and the genus is adjunction. The restriction sequence $0\to\OO_X\to\OO_X(H)\to\OO_C(H)\to0$ has cohomology
\[
\begin{aligned}
0\to H^0(\OO_X)\to H^0(\OO_X(H))&\to H^0(\OO_C(H))\to H^1(\OO_X)\\
&\to H^1(\OO_X(H))\to H^1(\OO_C(H))\to H^2(\OO_X),
\end{aligned}
\]
where $H^1(\OO_X)=0$ because $q=0$ and $H^2(\OO_X)=0$ because $p_g=0$. Hence $H^1(\OO_X(H))\cong H^1(C,\OO_C(H))$ and $h^0(\OO_C(H))=h^0(\OO_X(H))-1$. By adjunction $K_C=(K_X+H)|_C$, so Serre duality on $C$ gives $h^1(\OO_C(H))=h^0(K_C-H|_C)=h^0(K_X|_C)$.

Assume $h^1(\OO_X(H))\neq0$, so that $\OO_C(H)$ is a special divisor class on $C$ of degree $H^2>0$. Clifford's theorem gives $h^0(\OO_C(H))\le\tfrac12H^2+1$, i.e.\ $h^0(\OO_X(H))\le\tfrac12H^2+2$. On the other hand $h^2(\OO_X(H))=0$ (Conventions), hence $h^0(\OO_X(H))=1+\tfrac12H\cdot(H-K_X)+h^1(\OO_X(H))$. Combining the two displays yields the bound. Equality in Clifford's theorem forces $\OO_C(H)$ to be trivial, or canonical, or $C$ to be hyperelliptic; the first is excluded by $H^2>0$, and the second means $K_X|_C=K_C-H|_C\cong\OO_C$.
\end{proof}

\begin{remark}\label{rem:restriction}
(i) Lemma~\ref{lem:restriction} identifies speciality of the polarization with speciality of the hyperplane section, which is the classical notion for surfaces in projective space: following Alexander \cite{Al88,Al92}, a surface $X\subseteq\PP^N$ is called \emph{special} when $h^1(\OO_X(H))\neq0$.

(ii) The bound is one-sided in a useful way: a special polarization forces $H\cdot K_X\ge0$. Together with Kawamata--Viehweg this closes several familiar families at once: if $H-K_X$ is nef and big then $h^1(\OO_X(H))=0$, which happens whenever $-K_X$ is nef and whenever $K_X\equiv0$, so del Pezzo surfaces, rational elliptic surfaces, Enriques surfaces and (for $q=1$) bielliptic surfaces are never special, for any polarization. Finally $H\cdot K_X=0$ with $H$ ample forces $K_X^2\le0$ by Hodge index, so on a minimal surface of general type the equality case above can only occur with $C$ hyperelliptic.

(iii) Still assuming $p_g=0$ but allowing $q>0$, the identity degenerates into $h^1(C,\OO_C(H))\le h^1(\OO_X(H))\le q+h^1(C,\OO_C(H))$, the exact value being $q-\operatorname{rk}\delta+h^1(\OO_C(H))$ with $\delta:H^0(\OO_C(H))\to H^1(\OO_X)$ the connecting map. In particular an irregular surface can be special with non-special hyperplane sections, which is why Theorem~\ref{thm:nogo} is stated for $q\ge0$ at no extra cost.
\end{remark}

\subsection{The hypothesis is not vacuous: special rational surfaces}\label{subsec:special-rational}
Since Corollary~\ref{cor:burniat-nonspecial} below shows that primary Burniat surfaces carry no special polarization at all, it is worth recording that special very ample divisors on surfaces with $p_g=q=0$ do exist, and are classical. Alexander classified the non-special linearly normal rational surfaces in $\PP^4$, showing that for each degree $3\le d\le9$ they form a single irreducible family \cite{Al88}, and then the rational surfaces in $\PP^4$ of speciality one \cite{Al92}; these have $8\le d\le11$. Rational surfaces have $p_g=q=0$, so every one of them is a pair to which Theorem~\ref{thm:nogo} applies. The smallest is the following.

\begin{example}\label{ex:special-rational}
Let $x_1,\dots,x_4,y_5,\dots,y_{16}$ be $16$ points of $\PP^2$ and let $S$ be their blow up, with $L$ the pullback of a line. We write $x_i$ and $y_k$ also for the classes of the exceptional divisors over the corresponding points, and set
\[
H\;\equiv\;6L-2\sum_{i=1}^{4}x_i-\sum_{k=5}^{16}y_k .
\]
For a suitable configuration of the points --- constructed by Okonek \cite{Ok86} via reflexive sheaves, and again, directly from the linear system, by Catanese and Hulek \cite[Thm.~III.17]{CH97}, who moreover determine exactly which configurations occur \cite[Thm.~III.14]{CH97} and show that the moduli space is $\mathcal M/S_5$ with $\mathcal M$ rational of dimension $19$ \cite[Thm.~III.20]{CH97} --- the system $|H|$ embeds $S$ in $\PP^4$ as a linearly normal surface of degree $8$ and sectional genus $6$. One computes
\[
H^2=8,\quad H\cdot K_S=2,\quad K_S^2=-7,\quad \chi(\OO_S(H))=1+\tfrac12H\cdot(H-K_S)=4,
\]
and $h^2(\OO_S(H))=h^0(\OO_S(K_S-H))=0$, since $H\cdot(K_S-H)=-6<0$ with $H$ ample, while $h^0(\OO_S(H))=5$, the embedding being linearly normal by \cite[Thm.~III.17]{CH97}: the restriction $H^0(\PP^4,\OO(1))\to H^0(\OO_S(H))$ is bijective. Hence
\[
h^1\bigl(\OO_S(H)\bigr)=5-4=1\;\neq\;0 .
\]
Thus $(S,H)$ satisfies the hypotheses of Theorem~\ref{thm:nogo}, and consequently $S$ carries no rank two special Ulrich bundle with respect to $H$ in an extension of the adjoint-kernel shape \eqref{eq:shape}, for any $0$-dimensional $Z\subset S$. The Ulrich numerology reads
\[
\begin{gathered}
\len(Z)=\chi(\OO_S(H))+1=5,\qquad c_2=25,\\
c_1=3H+K_S=15L-5\textstyle\sum_{i=1}^{4} x_i-2\sum_{k=5}^{16} y_k .
\end{gathered}
\]
\end{example}

\begin{remark}\label{rem:special-rational-context}
(i) None of the known existence criteria applies to $(S,H)$. Theorem~\ref{thm:main} requires $h^1(\OO_S(H))=0$; Beauville's hypothesis fails too, since by the implication recorded in \S\ref{subsec:mainresults} an irreducible member of $|H-K_S|$ would force $h^1(\OO_S(H))=0$. Note also that $-K_S$ is not effective, as $H\cdot(-K_S)=-2<0$ with $H$ ample; in particular $S$ is not an anticanonical rational surface. Whether $(S,H)$ carries a special Ulrich bundle of rank two at all is therefore open, and is the smallest concrete instance of Question~\ref{q:twisted}.

(ii) $(S,H)$ is a genuine polarization, unlike the classes $H_u$ of Proposition~\ref{prop:ray} on a primary Burniat surface, which are ample but not very ample. The twisted-kernel program of \S\ref{subsec:twisted} can therefore be tested on it in the embedding sense, and the geometry is favourable: by \cite[Lem.~III.7, Cor.~III.11]{CH97} the hyperplane class restricts to the \emph{canonical} series on the general member of the pencil $|H-x_i|=|D_i|$, a curve of genus $4$ with $H$-degree $6$, which is precisely where the unit of speciality is carried.

(iii) Alexander's speciality one surfaces in $\PP^4$ of degree at most $10$ are blow ups of $\PP^2$ at $13$, $15$ or $16$ points, listed explicitly in \cite[\S IV]{CH97}. Higher speciality also occurs: the two families of degree $10$ and sectional genus $9$ on $18$ points have $h^1(\OO_S(H))=2$. In general, for a linearly normal $X\subseteq\PP^N$ with $h^2(\OO_X(H))=0$ one has $h^1(\OO_X(H))=\pi(H)-H^2+N-\chi(\OO_X)$, so for a rational surface in $\PP^4$ speciality is the failure of $\pi(H)\le H^2-3$; Example~\ref{ex:special-rational} has $\pi=6=H^2-2$, one unit above, and Lemma~\ref{lem:restriction} allows $h^1\le\tfrac12H\cdot K_S+1=2$. The finiteness of the list is a theorem of Ellingsrud and Peskine \cite{EP89}.
\end{remark}

\section{Higher rank and Ulrich wildness}\label{sec:evenrank}

Throughout this section $X$ is as in Theorem~\ref{thm:main} and, in addition, minimal of non-negative Kodaira dimension; then $K_X$ is nef, so $H\cdot K_X\ge0$, and $-K_X$ is not effective (otherwise $-mK_X$ would be effective for the $m\ge1$ with $mK_X$ effective, so that $K_X$ is torsion, and a torsion class with $-K_X$ effective satisfies $K_X\sim0$, giving $p_g=1$). Recall from \cite{FPL,CasWild} that $(X,H)$ is \emph{Ulrich wild} if the category of Ulrich sheaves on $(X,H)$ is of wild representation type, and \emph{strictly} so if this is witnessed by an exact \emph{fully faithful} functor from the finite-dimensional representations of a wild algebra --- for us the $d$-arrow Kronecker quiver with $d\ge3$. Full faithfulness is what makes the notion strict and is what we use: it carries bricks to \emph{simple} objects, so that the resulting families consist of simple Ulrich bundles and not merely indecomposable ones, and it is in that concrete form that Theorem~\ref{thm:evenrank} is stated.

As recalled in \S\ref{subsec:mainresults}, wildness itself is already known under exactly the hypothesis of Theorem~\ref{thm:evenrank}(b): the condition $H^2+4-K_X^2\ge3$ is the condition $H^2+1\ge K_X^2$ of \cite[Lem.~5.2]{Cas17}, whose other hypothesis $\pi(H)\ge1$ is automatic here since $K_X$ is nef, and whose proof is the same $\Ext$ estimate fed to the same criterion of \cite{FPL} for the same orthogonal pair of stable rank two special Ulrich bundles; \cite[Thm.~1.3]{Cas17} dispenses with the condition on $H^2-K_X^2$ altogether. What Theorem~\ref{thm:evenrank} adds is four things: \emph{simplicity} rather than indecomposability, from the full faithfulness of the functor of \cite{FPL} and not merely from its exactness; control of the rank, always even; an explicit dimension count; and independence from \cite[Thm.~1.2]{Cas17}, Step~1 producing the orthogonal pair by hand instead of quoting stability for a general $0$-cycle. The first three rest on the injectivity of $Z\mapsto\EE_Z$, established next. We write $\EE_Z$ for the bundle produced by Theorem~\ref{thm:main} from $Z$, uniquely determined by $Z$ by Step~2 of that proof; the lemma below shows that conversely a generic $Z$ is determined by $\EE_Z$.

\begin{lemma}\label{lem:noniso}
In the situation above,
\[
h^0\bigl(\OO_X(H-K_X)\bigr)\;\le\;\len-1 ,
\]
where $\len=\chi(\OO_X(H))+1$ is the Ulrich length. Consequently, if the generic reduced $Z\in\Hilb^{\len}(X)$ of Step~1 is chosen to satisfy in addition the (generic, by Lemma~\ref{lem:LQ24}) condition $h^0(\OO_X(H-K_X)\otimes\II_Z)=0$, then
\[
h^0\bigl(\EE_Z(-H-K_X)\bigr)=1 ,
\]
the zero scheme of the unique section of $\EE_Z(-H-K_X)$ recovers $Z$, and the assignment $Z\mapsto\EE_Z$ is injective on a dense open subset of $\Hilb^{\len}(X)$, a scheme of dimension $2\len$. Only injectivity on closed points is used below; we make no moduli-theoretic assertion about the family. Cf.\ \cite[Rem.~4.5]{Cas17}, where the same injectivity is obtained under the hypothesis $h^0(\OO_X(H-K_X))=0$, which the bound above replaces, at no cost, by a genericity condition on $Z$.
\end{lemma}

\begin{proof}
For the bound we may assume $h^0(\OO_X(H-K_X))>0$. Let $C\in|H|$ be a general member: since $H$ is very ample, $C$ is a smooth irreducible curve, of genus $g$ with $g-1=\tfrac12H\cdot(H+K_X)$. Since $-K_X$ is not effective, the restriction sequence
\[
0\to\OO_X(-K_X)\to\OO_X(H-K_X)\to\OO_C(D)\to0,\qquad D:=(H-K_X)|_C ,
\]
embeds $H^0(\OO_X(H-K_X))$ into $H^0(\OO_C(D))$; here $\deg D=H\cdot(H-K_X)=2\len-4$. If $h^1(\OO_C(D))>0$, Clifford's theorem gives $h^0(\OO_C(D))\le\tfrac12\deg D+1=\len-1$. If $h^1(\OO_C(D))=0$, Riemann--Roch on $C$ gives
\[
h^0(\OO_C(D))=\deg D+1-g=\tfrac12H\cdot(H-K_X)-H\cdot K_X\le\len-2 ,
\]
using $H\cdot K_X\ge0$. In all cases $h^0(\OO_X(H-K_X))\le\len-1<\len$, so Lemma~\ref{lem:LQ24} applies with $|M|=|H-K_X|$: the condition $h^0(\OO_X(H-K_X)\otimes\II_Z)=0$ is generic and can be imposed alongside the (finitely many, generic) conditions of Step~1.

Now twist \eqref{eq:shape} by $-H-K_X$ and take cohomology; since $q=0$, the connecting map lands in $H^1(\OO_X)=0$ and the sequence
\[
0\to H^0(\OO_X)\to H^0(\EE_Z(-H-K_X))\to H^0(\OO_X(H-K_X)\otimes\II_Z)\to 0
\]
is exact,
so $h^0(\EE_Z(-H-K_X))=1$. The unique section, up to scalar, is the one defining the extension \eqref{eq:shape}, and its zero scheme is exactly $Z$. If $\EE_Z\cong\EE_{Z'}$ with both cycles chosen as above, transporting the section defining $Z'$ yields a section of $\EE_Z(-H-K_X)$ with zero scheme $Z'$; by uniqueness of the section, $Z'=Z$. Injectivity on the intersection of the generic conditions follows, and $\dim\Hilb^{\len}(X)=2\len$.
\end{proof}

\begin{lemma}[Balanced bricks for the $d$-Kronecker quiver]\label{lem:bricks}
Let $d\ge3$ and let $a,b\ge1$ satisfy $|a-b|\le1$. In the affine space $\Rep_d(a,b)=\Hom(\CC^a,\CC^b)^{\oplus d}$ of representations of the $d$-arrow Kronecker quiver of dimension vector $(a,b)$, the bricks --- those $\rho$ with $\End(\rho)=\CC$ --- form a dense open subset, and their isomorphism classes form a family of dimension
\[
dab-a^2-b^2+1\;=\;1-q_d(a,b)\;\ge\;1 .
\]
\end{lemma}

\begin{proof}
The function $\rho\mapsto\dim\End(\rho)$ is upper semicontinuous on the irreducible affine space $\Rep_d(a,b)$ and is everywhere $\ge1$, so the brick locus is open, and dense as soon as it is non-empty. For non-emptiness note that setting $f_{k}=0$ for $k>3$ does not enlarge the endomorphism algebra, the corresponding equations becoming vacuous; so it is enough to exhibit one brick using three arrows, or two.

Suppose first $a=b=n$ and take $f_1=\mathrm{id}$, $f_2=\mathrm{diag}(\lambda_1,\dots,\lambda_n)$ with the $\lambda_i$ pairwise distinct, and $f_3=J$, the matrix all of whose entries equal $1$. An endomorphism is a pair $(\varphi,\psi)$ of endomorphisms of $\CC^n$ with $\psi f_i=f_i\varphi$ for every $i$. The equation for $f_1$ gives $\psi=\varphi$; the equation for $f_2$ then says that $\varphi$ commutes with a regular diagonal matrix, so $\varphi=\mathrm{diag}(\mu_1,\dots,\mu_n)$; and the equation for $f_3$ reads $\mu_i=\mu_j$ for all $i,j$, since $(\varphi J)_{ij}=\mu_i$ while $(J\varphi)_{ij}=\mu_j$. Hence $\varphi$ is scalar.

Suppose next $\{a,b\}=\{n,n+1\}$ and take $f_k=0$ for $k>2$, with $(f_1,f_2)$ an indecomposable representation of the $2$-arrow Kronecker quiver of dimension vector $(a,b)$. Such a representation exists, is unique up to isomorphism and is a brick, because $q_2(a,b)=a^2+b^2-2ab=1$ makes $(a,b)$ a real root.

For the dimension count, the stabiliser of a brick in $GL_a\times GL_b$ is exactly the subgroup of scalars, so every orbit contained in the brick locus has dimension $a^2+b^2-1$. The brick locus is irreducible of dimension $dab$ and all of its orbits have this same dimension, so by Rosenlicht's theorem it contains a dense open invariant subset admitting a geometric quotient; that quotient is the family of isomorphism classes of the statement, and its dimension is $dab-a^2-b^2+1$. Finally $q_d(a,b)\le a^2+b^2-3ab<0$ for $a,b\ge1$ with $|a-b|\le1$, so $1-q_d(a,b)\ge1$.
\end{proof}

\begin{remark}
Lemma~\ref{lem:bricks} is a special case of the general theory: for a quiver without oriented cycles a dimension vector is a Schur root exactly when the general representation of that dimension is a brick, and the Schur roots are described by Kac \cite{Kac80} and Schofield \cite{Sch92}. We give the direct argument because only balanced dimension vectors are needed, and there an explicit brick is immediate.
\end{remark}

\begin{proof}[Proof of Theorem \textup{\ref{thm:evenrank}}]
By Remark~\ref{rem:rank2forced} we do not iterate the extension \eqref{eq:shape} directly; instead we produce an \emph{orthogonal pair} of rank two special Ulrich bundles and feed it to the wildness criterion of Faenzi--Pons-Llopis \cite{FPL}.

\smallskip
\noindent \emph{Step 1: an orthogonal pair.} Lemma~\ref{lem:noniso} provides an uncountable family $\{\EE_Z\}$ of pairwise non-isomorphic special Ulrich bundles of rank two, each $\mu_H$-semistable (Lemma~\ref{lem:semistable}). We claim the family contains two members $\EE,\EE'$ which are simple and satisfy
\begin{equation}\label{eq:orth}
\Hom(\EE,\EE')=\Hom(\EE',\EE)=0 .
\end{equation}
If uncountably many members are $\mu_H$-stable, pick two non-isomorphic stable ones: stable bundles are simple, and non-isomorphic stable bundles of equal slope satisfy \eqref{eq:orth}. Otherwise uncountably many members are strictly $\mu_H$-semistable, hence extensions of Ulrich line bundles $M_1,M_2$ with $M_1+M_2\equiv3H+K_X$ (Lemma~\ref{lem:semistable}); since $q=0$, $\Pic(X)$ is discrete and countable, so some fixed ordered pair $(M_1,M_2)$ occurs for an uncountable subfamily. As the split extension forms a single isomorphism class, we may take all these extensions
\[
0\to M_1\to\EE_Z\to M_2\to0
\]
non-split. The case $M_1\cong M_2$ is excluded numerically: $2M_1\equiv3H+K_X$ and $c_2(\EE)=M_1^2$ force $K_X^2=H^2+8$, contradicting $H^2+4>K_X^2$. For $M_1\not\cong M_2$ of equal $H$-degree we have $\Hom(M_i,M_j)=0$ for $i\neq j$: the difference is a non-trivial class of $H$-degree zero, hence not effective. Recall the standard compatibility: a morphism between two such extensions, with classes $e,e'\in\Ext^1(M_2,M_1)$, induces scalars $\lambda$ on $M_1$ and $\mu$ on $M_2$ (as $\Hom(M_1,M_2)=0$ and the $M_i$ are simple), subject to $\lambda e=\mu e'$ in $\Ext^1(M_2,M_1)$. For an endomorphism $\varphi$ of a fixed non-split $\EE$ this reads $\lambda e=\mu e$ with $e\neq0$, so $\mu=\lambda$; then $\varphi-\lambda\,\mathrm{id}$ kills $M_1$, factors through a map $M_2\to\EE$ which composes to zero with $\EE\to M_2$ (otherwise the extension splits), hence lands in $M_1$ and vanishes: $\EE$ is simple. For $\EE\not\cong\EE'$ the classes $e,e'$ are non-proportional, so $\lambda e=\mu e'$ forces $\lambda=\mu=0$ and any $\varphi:\EE\to\EE'$ factors as $\EE\to M_2\to M_1\subset\EE'$, hence vanishes. This proves \eqref{eq:orth} in all cases.

\smallskip
\noindent \emph{Step 2: the extension space.} For any two bundles $\EE,\EE'$ of rank two with the Chern classes \eqref{eq:c2special}, Riemann--Roch gives
\[
\chi(\EE'^\vee\otimes\EE)\;=\;4\chi(\OO_X)+c_1^2-4c_2\;=\;K_X^2-H^2-4 ,
\]
whence, for the orthogonal pair of Step~1,
\[
\dim\Ext^1(\EE',\EE)\;\ge\;-\chi(\EE'^\vee\otimes\EE)\;=\;H^2+4-K_X^2\;=:\;w ;
\]
cf.\ \cite[Lem.~5.2]{Cas17} for the same estimate.

\smallskip
\noindent \emph{Step 3: rank four.} If $H^2+4>K_X^2$ then $w\ge1$ and there is a non-split extension $0\to\EE\to\FF\to\EE'\to0$. Since the class of Ulrich bundles is closed under extensions (immediate from \eqref{eq:ulrich-vanishing}), $\FF$ is an Ulrich bundle of rank four, and it is simple by the diagram chase of Step~1, using that $\EE,\EE'$ are simple and mutually orthogonal. This proves (a).

\smallskip
\noindent \emph{Step 4: wildness.} Assume $w\ge3$, and set
\[
d\;:=\;\dim\Ext^1(\EE',\EE)\;\ge\;w\;\ge\;3 ,
\]
by Step~2. The pair $(\mathcal{A},\mathcal{B})=(\EE,\EE')$ consists of simple Ulrich bundles with $\Hom(\mathcal{A},\mathcal{B})=\Hom(\mathcal{B},\mathcal{A})=0$ and $\dim\Ext^1(\mathcal{B},\mathcal{A})=d\ge3$; by the criterion of Faenzi--Pons-Llopis \cite[Thm.~A and Cor.~2.1]{FPL}, $X$ is strictly Ulrich wild. The functor $\Phi$ of \cite[\S2]{FPL} sends a representation of the $d$-arrow Kronecker quiver of dimension vector $(a,b)$ to an Ulrich bundle of rank $2(a+b)$ filtered by $a$ copies of $\mathcal{A}$ and $b$ copies of $\mathcal{B}$, and is fully faithful; bricks therefore map to \emph{simple} Ulrich bundles, which is where the present construction improves on the indecomposability of \cite[Thm.~1.3]{Cas17}. For every even $r=2(a+b)\ge4$ choose $a,b\ge1$ balanced, $|a-b|\le1$; then the Tits form
\[
q_d(a,b)\;=\;a^2+b^2-dab\;\le\;a^2+b^2-3ab\;<\;0 ,
\]
so $q_d(a,b)<0$. By Lemma~\ref{lem:bricks} the bricks of dimension vector $(a,b)$ form a dense open subset of the representation space, and their isomorphism classes form a family of dimension
\[
dab-(a^2+b^2)+1\;=\;1-q_d(a,b) ,
\]
which grows quadratically with $r$. Full faithfulness of $\Phi$ \cite[Thm.~A]{FPL} carries bricks to simple Ulrich bundles. Their images under $\Phi$ are the required families of pairwise non-isomorphic simple Ulrich bundles of rank $r$; the case $r=2$ is Step~1, where the members of the family of Lemma~\ref{lem:noniso} are shown to be simple --- stable, hence simple, in the first branch, and simple by the non-split extension chase in the second. This proves (b). If $X$ is minimal of general type with $p_g=q=0$ then $K_X^2\le9$ by Bogomolov--Miyaoka--Yau, so $w\ge3$ holds as soon as $H^2\ge8$. The phenomenon is an instance of Ulrich wildness, cf.\ \cite{FPL,CasWild}.
\end{proof}

\begin{remark}[Why only even ranks]\label{rem:evenonly}
Only ranks $2(a+b)$ occur, and this is structural: the input of \cite[Thm.~A]{FPL} is an orthogonal \emph{pair} of Ulrich bundles, here of rank two, and $\Phi$ sends a representation of dimension vector $(a,b)$ to a bundle filtered by $a$ copies of $\mathcal A$ and $b$ copies of $\mathcal B$. An odd rank would need an Ulrich bundle of odd rank to start from --- for rank one, an Ulrich line bundle, which a primary Burniat surface does not have \cite{Cho26} --- or a different mechanism. This is no defect relative to \cite[Thm.~1.3]{Cas17}, which produces arbitrarily large families of indecomposable Ulrich bundles but not, as far as we can see, a prescribed rank; what is restricted to even rank here is the finer conclusion, simplicity together with the dimension count.
\end{remark}

\section{Primary Burniat surfaces}\label{sec:burniat}

The comparison with the recent work of Cho \cite{Cho26} was stated in \S\ref{subsec:mainresults} and governs the layout of this section. Subsection~\ref{subsec:nonspecial} proves the global non-speciality statement, Theorem~\ref{thm:burniat}(a), in a few lines from the degree bound of Corollary~\ref{cor:degreetwo}, itself proved there from Lemma~\ref{lem:bpf-degree} and coinciding with \cite[Prop.~5.4]{Cho26}. Subsection~\ref{subsec:rays} then classifies the special ample classes, Theorem~\ref{thm:burniat}(b); that classification is of independent interest and is where the eigensheaf technique is set up, but it is \emph{not} a step in the proof of~(a), and the reader interested only in the Ulrich consequences may skip it. Subsection~\ref{subsec:twisted} proves~(c). Attribution is recorded locally, at each statement which overlaps \cite{Cho26} or \cite{Ale16}.

\subsection{The surface, its lattice, and the Burniat test}\label{subsec:burniat}
We now carry out the computation which first revealed Corollary~\ref{cor:genus}, on the most symmetric surfaces of general type with $p_g=q=0$ that come equipped with a supply of natural curves: primary Burniat surfaces. We follow \cite{BC11,Pet77}.

Let $p_1,p_2,p_3\in\PP^2$ be non-collinear points and let $Y\to\PP^2$ be the blow-up at the three points, a del Pezzo surface of degree $6$, with $\Pic(Y)=\ZZ\langle e_0,e_1,e_2,e_3\rangle$ ($e_0$ the pullback of a line, $e_i$ the exceptional curves), $-K_Y=3e_0-e_1-e_2-e_3$. A \emph{primary Burniat surface} is a bidouble cover $\pi:X\to Y$, Galois with group $(\ZZ/2)^2$, branched over the strict transforms of the nine Burniat lines (the three sides of the triangle, in classes $e_0-e_i-e_j$, plus two further lines through each $p_i$, in classes $e_0-e_i$) together with the three exceptional curves $e_1,e_2,e_3$. The total branch class is
\[
\Delta\equiv\textstyle\sum_{i<j}(e_0-e_i-e_j)+\sum_i2(e_0-e_i)+\sum_ie_i=9e_0-3(e_1+e_2+e_3)=-3K_Y .
\]
$X$ is a minimal surface of general type with $p_g=q=0$, $K_X$ ample, $K_X^2=6$.

\smallskip
\noindent \emph{The lattice.} Each branch component $\Gamma\subset Y$ has reduced preimage $R(\Gamma)\subset X$ with $\pi^*\Gamma=2R(\Gamma)$, so $R(\Gamma)=\tfrac12\pi^*[\Gamma]$ in $\Num(X)_\mathbb{Q}$. Writing
\[
f_i:=\tfrac12\pi^*e_i=[R(E_i)],\qquad
h:=\tfrac12\pi^*e_0=[R(\text{line through }p_i)]+f_i ,
\]
both classes are \emph{integral} (differences and sums of curve classes). Since $\pi$ has degree $4$,
\[
\bigl(\tfrac12\pi^*a\bigr)\cdot\bigl(\tfrac12\pi^*b\bigr)=a\cdot b ,
\]
so $\ZZ\langle h,f_1,f_2,f_3\rangle$ is a unimodular lattice of signature $(1,3)$ inside $\Num(X)$. Now $c_2(X)=12\chi(\OO_X)-K_X^2=6$, hence $b_2(X)=4$ and (as $p_g=0$) $\rho(X)=4$. The ambient lattice $\Num(X)$ is itself unimodular: since $p_g=0$ every class of $H^2(X,\ZZ)$ is of type $(1,1)$, so the Lefschetz theorem on $(1,1)$-classes gives $\Num(X)=H^2(X,\ZZ)/\mathrm{tors}$, which is unimodular by Poincar\'e duality. A full-rank unimodular sublattice of a unimodular lattice has index one, hence is the whole lattice, so
\[
\Num(X)=\ZZ\langle h,f_1,f_2,f_3\rangle\cong\langle1\rangle\oplus\langle-1\rangle^{\oplus3}.
\]
Moreover $K_X=\pi^*K_Y+\sum R(\Gamma)=\pi^*K_Y+\tfrac12\pi^*(-3K_Y)=\tfrac12\pi^*(-K_Y)$, i.e.,
\[
K_X=3h-f_1-f_2-f_3,\qquad 2K_X=\pi^*(-K_Y),
\]
recovering the fact that the bicanonical map of $X$ is the Galois cover $\pi$ followed by the anticanonical embedding $Y\hookrightarrow\PP^6$.

\smallskip
\noindent \emph{The natural curves.} The reduced preimages of the twelve branch components give twelve irreducible curves on $X$, with invariants computed by adjunction on $X$:
\begin{center}
\begin{tabular}{lcccc}
\toprule
curve & class & $C^2$ & $K_X\cdot C$ & $p_a(C)$\\
\midrule
$\varepsilon_i=R(E_i)$ & $f_i$ & $-1$ & $1$ & $1$\\
$\sigma_{ij}=R(\overline{p_ip_j})$ & $h-f_i-f_j$ & $-1$ & $1$ & $1$\\
$\lambda_i^{(1)},\lambda_i^{(2)}=R(\text{line}\ni p_i)$ & $h-f_i$ & $0$ & $2$ & $2$\\
\bottomrule
\end{tabular}
\end{center}

\smallskip
\noindent \emph{The test.} Every curve in the table has $p_a\ge1$. By Corollary~\ref{cor:genus}, for \emph{every} ample $H$ on $X$ and every curve $C$ in the table satisfying the B\'ezout bound \eqref{eq:star-b} we have $h^0(\OO_X(H-C))\ge p_a(C)\ge1$: the residual system is never empty, and the points-on-a-curve mechanism for the adjoint kernel fails on all twelve natural curves --- as it must, by Theorem~\ref{thm:nogo}. (A hypothetical rational curve on $X$ would not help either: by Corollary~\ref{cor:genus} it could only certify polarizations that were already non-special.)

\begin{remark}\label{rem:burniat-positive}
Along the canonical ray Theorem~\ref{thm:main} applies at once: $K_X$ is ample with $K_X^2=6$, so $mK_X$ is very ample for $m\ge3$ by Reider's theorem (the exceptional configurations are excluded using $K_X\cdot D\ge1$, parity in the adjunction formula and $D^2K_X^2\le(K_X\cdot D)^2$), while $h^1(\OO_X(mK_X))=0$ for $m\ge2$ by Kawamata--Viehweg. Hence primary Burniat surfaces carry special Ulrich bundles for all pluricanonical polarizations, consistently with \cite{Cas22}; and since they carry no Ulrich line bundle at all \cite{Cho26}, these bundles are $\mu$-stable unconditionally by Corollary~\ref{cor:stable}.
\end{remark}

\begin{remark}[Cho's bundle]\label{rem:cho}
For $H=3K_X$ Cho \cite[Thm.~5.6]{Cho26} constructs a rank two Ulrich bundle $\mathcal{V}$ from an extension whose shape differs from \eqref{eq:shape}. The Chern classes do not separate it from the bundles of Theorem~\ref{thm:main}: his computation gives $\det\mathcal{V}=4K_X+2H=10K_X=3H+K_X$, so $\mathcal{V}$ is \emph{special}, and $c_2$ is then determined by \eqref{eq:c2special}. What does separate them is the invariant $h^0(\EE(-H-K_X))$, which is $\ge1$ for any bundle arising from \eqref{eq:shape} and is $0$ for $\mathcal{V}$; both that computation and the conclusion that $\mathcal V$ is none of the bundles of \cite{Cas17} are already in \cite[Rem.~5.7]{Cho26}. What we add is that the Chern classes do not separate the two constructions, so that the invariant above is what does the work, and that the comparison can be made with the whole family at once --- unconditionally, one bundle for each admissible $0$-cycle of length $\len(Z)=\chi(\OO_X(3K_X))+1=20$. Whether the two constructions land in the same irreducible component of the moduli space is raised as a question in \cite[Rem.~5.7]{Cho26}, and we do not know the answer.
\end{remark}

\subsection{Every ample and base point free divisor is non-special}\label{subsec:nonspecial}
We keep the notation of \S\ref{subsec:burniat}: $\pi:X\to Y$ is the bidouble cover of the del Pezzo surface $Y$ of degree $6$, $\Num(X)=\ZZ\langle h,f_1,f_2,f_3\rangle$, and $K_X=3h-f_1-f_2-f_3$. Only one geometric input is needed: an ample and base point free divisor on $X$ has degree at least two on each of the six elliptic ramification curves (Corollary~\ref{cor:degreetwo} below, proved from Lemma~\ref{lem:bpf-degree}; the same bound was obtained independently in \cite[Prop.~5.4]{Cho26}). Everything else is the description of $\NE(X)$, which we record first because it is used again in \S\ref{subsec:rays}.

The Mori cone of $X$ is completely determined by that of $Y$, by an elementary pushforward argument.

\begin{lemma}\label{lem:mori}
Let $\pi:X\to Y$ be a primary Burniat bidouble cover and let $\iota:\Num(Y)\to\Num(X)$, $a\mapsto\tfrac12\pi^*a$, be the map of \S\ref{subsec:burniat}; recall from there that $\iota$ is an isometric \emph{isomorphism}, since $\ZZ\langle h,f_1,f_2,f_3\rangle$ is a full-rank unimodular sublattice of the unimodular lattice $\Num(X)$. Then for every class $D\in\Num(X)$,
\[
\pi_*D\;=\;2\,\iota^{-1}(D).
\]
Consequently the Mori cone of $X$ is
\[
\NE(X)\;=\;\sum_{i}\RR_{\ge0}\,[\varepsilon_i]\;+\;\sum_{i<j}\RR_{\ge0}\,[\sigma_{ij}] ,
\]
the cone spanned by the six curves of the table in \S\ref{subsec:burniat} with $(C^2,K_X\cdot C)=(-1,1)$; the six curves $\lambda_i^{(k)}$ are redundant, since $\lambda_i\equiv\sigma_{ij}+\varepsilon_j$.
\end{lemma}

\begin{proof}
For $a\in\Num(Y)$ the projection formula gives $\pi_*D\cdot a=D\cdot\pi^*a=\iota(\iota^{-1}D)\cdot 2\iota(a)=2\,\iota^{-1}(D)\cdot a$, since $\iota$ is an isometry; as the intersection form on $Y$ is non-degenerate, $\pi_*D=2\iota^{-1}(D)$.

Now let $C\subset X$ be an irreducible curve. Then $\pi_*C=\deg(C/\pi(C))\cdot[\pi(C)]$ is a positive multiple of an irreducible curve class, hence effective on $Y$; by the displayed identity $\iota^{-1}[C]=\tfrac12\pi_*C$ lies in $\Eff(Y)_\QQ$. Since $Y$ is a del Pezzo surface of degree $6$, its effective cone is spanned by its six $(-1)$-curves $e_1,e_2,e_3$ and $e_0-e_i-e_j$ --- equivalently, $Y$ is the toric surface of the hexagon and these six curves are its torus-invariant prime divisors --- and their images under $\iota$ are precisely $[\varepsilon_i]$ and $[\sigma_{ij}]$. Hence $[C]$ lies in the displayed cone, which proves ``$\subseteq$''; the reverse inclusion holds because the six classes are classes of actual curves.
\end{proof}

\begin{remark}\label{rem:mori-general}
The proof used no feature of the Burniat geometry beyond the finiteness of $\pi$ and the fact that $\pi^*$ identifies the two N\'eron--Severi groups rationally: \emph{if $\pi:X\to Y$ is a finite surjective morphism of smooth projective surfaces with $\rho(X)=\rho(Y)$, then $\NE(X)_\QQ=\pi^*\NE(Y)_\QQ$}, by the same two lines. Here and below curve classes are identified with divisor classes through the intersection pairing, which is non-degenerate on $\Num$; ``$\pi^*$ of a curve class'' means the pullback of the corresponding divisor class under that identification, and it is in this sense that the displayed equality is asserted. Pushing forward, unlike pulling back, is insensitive to the sign of $C^2$, and this is what excludes a priori ``exotic'' negative curves on $X$ outside the cone spanned by the six ramification classes; the rational statement alone does not, however, control \emph{integral} intersection numbers upstairs, see Remark~\ref{rem:integrality}. The identification of the two effective cones is not new: it is the starting observation of Alexeev \cite[\S4]{Ale16}, who determines the semigroup of effective \emph{integral} divisors together with the nef cone; see also \cite{AO13}. The immediate consequence that a divisor of non-negative degree on the six curves is nef, and ample if the degrees are positive, is \cite[Prop.~4.2]{Cho26}.
\end{remark}

\begin{lemma}\label{lem:bpf-degree}
Let $D$ be a base point free divisor on a smooth projective surface $S$ and let $C\subset S$ be a smooth irreducible curve of genus $g(C)\ge1$. Then $D\cdot C\neq1$.
\end{lemma}

\begin{proof}
Suppose $D\cdot C=1$ and let $V\subseteq H^0(C,\OO_C(D))$ be the image of the restriction map. Since $\OO_C(D)$ has degree one and $C$ is not rational, $h^0(C,\OO_C(D))\le1$: a pencil in $|\OO_C(D)|$ would have a base point free moving part, necessarily of degree one since a base point free pencil has positive degree, so the fixed part would be empty and the pencil would define a degree one morphism $C\to\PP^1$, that is, an isomorphism. On the other hand $V\neq0$, since $C$ is not contained in the base locus of $|D|$. Hence $\dim V=1$, so every section of $H^0(S,\OO_S(D))$ restricts on $C$ to a multiple of one section $s$, and all of them vanish at the single point of $\operatorname{div}(s)$: that point is a base point of $|D|$.
\end{proof}

\begin{lemma}\label{lem:six-elliptic}
Each of the six curves $\varepsilon_1,\varepsilon_2,\varepsilon_3,\sigma_{12},\sigma_{13},\sigma_{23}$ spanning $\NE(X)$ is smooth of genus one.
\end{lemma}

\begin{proof}
Each is the reduced ramification curve $R(\Gamma)$ over one of the six $(-1)$-curves of $Y$: $\Gamma=E_i$ for $\varepsilon_i$, and $\Gamma=\widetilde{s_k}$ for $\sigma_{ij}$ with $\{i,j,k\}=\{1,2,3\}$. Such a $\Gamma$ is a component of exactly one of the three branch divisors, say of $\Delta_\chi$, so that $R(\Gamma)\to\Gamma\cong\PP^1$ is a double cover branched exactly at the points where $\Gamma$ meets the other two branch divisors; these points are distinct, the branch configuration being nodal and $X$ smooth. Their number is four in each case: for $\Gamma=E_3\subset\Delta_1$ one has $E_3\cdot\Delta_2=1$ and $E_3\cdot\Delta_3=3$, while for $\Gamma=\widetilde{s_1}\subset\Delta_1$, of class $e_0-e_1-e_2$, one has $\widetilde{s_1}\cdot\Delta_2=1$ and $\widetilde{s_1}\cdot\Delta_3=3$; the remaining four cases follow by the symmetry of the configuration permuting $p_1,p_2,p_3$. A double cover of $\PP^1$ branched at four distinct points is smooth of genus one, in accordance with $p_a=1$ in the table of \S\ref{subsec:burniat}.
\end{proof}

\begin{corollary}\label{cor:degreetwo}
Let $H$ be an ample and base point free divisor on a primary Burniat surface $X$. Then $H\cdot C\ge2$ for each of the six curves $C$ of Lemma~\ref{lem:six-elliptic}; consequently $H\cdot K_X\ge12$.
\end{corollary}

\begin{proof}
$H$ is ample and $C$ is a curve, so $H\cdot C\ge1$; and $H\cdot C\neq1$ by Lemma~\ref{lem:bpf-degree}, $C$ being smooth of genus one by Lemma~\ref{lem:six-elliptic}. For the second assertion, the six classes add up to the canonical class in $\Num(X)$,
\[
\sum_{i}[\varepsilon_i]+\sum_{i<j}[\sigma_{ij}]
\;=\;\sum_i f_i+\Bigl(3h-2\sum_i f_i\Bigr)
\;=\;3h-f_1-f_2-f_3\;=\;K_X ,
\]
so that $H\cdot K_X=\sum_C H\cdot C\ge12$.
\end{proof}

\begin{remark}[Relation with \cite{Cho26}]\label{rem:cho-degreetwo}
Corollary~\ref{cor:degreetwo} is \cite[Prop.~5.4 and its proof]{Cho26}, obtained there independently, for the same six curves and with the same bound $H\cdot K_X\ge12$; the statement of that proposition names three of the six degrees and the other three appear in its proof. The mechanism is the same one read through Alexeev's symmetric coordinates: the residual point of a degree one intersection with an elliptic ramification curve is forced to be a prescribed $2$-torsion point, hence a base point. Cho establishes the bound in order to exclude Ulrich \emph{line} bundles \cite[Thm.~5.5]{Cho26}. What follows in this subsection is not in \cite{Cho26}.
\end{remark}

\begin{corollary}\label{cor:burniat-nonspecial}
Let $X$ be a primary Burniat surface. Then \emph{every} ample and base point free divisor on $X$ is non-special: $h^1(\OO_X(H))=0$. In particular this holds for every polarization $H$, and consequently:
\begin{enumerate}[label=\textup{(\alph*)}]
\item $(X,H)$ carries a special Ulrich bundle of rank two for every polarization $H$;
\item every such bundle is $\mu_H$-stable, $X$ carrying no Ulrich line bundle \cite[Thm.~5.5]{Cho26} --- a statement proved there, as is the first assertion here, for every ample and base point free divisor;
\item $X$ is strictly Ulrich wild with respect to \emph{every} polarization.
\end{enumerate}
\end{corollary}

\begin{proof}
Let $H$ be ample and base point free. By Corollary~\ref{cor:degreetwo}, $H\cdot C\ge2$ for each of the six curves, whereas $K_X\cdot C=1$ on each of them by the table of \S\ref{subsec:burniat}. Hence
\[
(H-K_X)\cdot C\;=\;H\cdot C-K_X\cdot C\;\ge\;1\;>\;0
\]
on all six. By Lemma~\ref{lem:mori} those six classes span $\NE(X)$, which is therefore polyhedral and closed, so Kleiman's criterion makes $H-K_X$ ample. Applying Kodaira vanishing to $H=K_X+(H-K_X)$ gives $h^1(\OO_X(H))=h^2(\OO_X(H))=0$. Every very ample divisor is base point free, so in particular every polarization is non-special.

Assertion (a) is then Theorem~\ref{thm:main}, which requires $H$ very ample; (b) follows from Corollary~\ref{cor:stable} together with \cite[Thm.~5.5]{Cho26}; and (c) is Theorem~\ref{thm:evenrank}(b), whose hypothesis $H^2+4-K_X^2\ge3$ reads $H^2\ge5$ since $K_X^2=6$ --- a condition which is in fact automatic for a polarization. Indeed $H$ very ample forces $h^0(\OO_X(H))\ge4$, since $|H|$ embeds $X$ as a non-degenerate surface in $\PP^{h^0-1}$ and a surface spanning $\PP^2$ is $\PP^2$, which $X$ is not; while $h^1=0$ by the first assertion and $h^2(\OO_X(H))=h^0(\OO_X(K_X-H))=0$. Hence $1+\tfrac12H\cdot(H-K_X)\ge4$, that is, $H^2-H\cdot K_X\ge6$, and Corollary~\ref{cor:degreetwo} gives $H\cdot K_X\ge12$; therefore
\[
H^2\;\ge\;H\cdot K_X+6\;\ge\;18 .
\]
\end{proof}

\begin{remark}[What is borrowed and what is not]\label{rem:cho-nonspecial}
Everything above is three lines of geometry on top of Corollary~\ref{cor:degreetwo}, which is in turn three lines on top of Lemma~\ref{lem:bpf-degree}. The step that is not in \cite{Cho26} is the one taken here: that the degree bound forces $H-K_X$ to be ample and hence makes \emph{every} polarization non-special, non-speciality being used there only for the single class $H=3K_X$, in order to invoke \cite{Cas17}. It is the quantification over all polarizations that turns Theorem~\ref{thm:main} into an unconditional existence statement on $X$ and, through Theorem~\ref{thm:evenrank}, into unconditional strict Ulrich wildness. The classification of \S\ref{subsec:rays} gives a second and independent route to the same conclusion, which is how we first found it; we have demoted it because the route above is shorter.

Three further points, for the record. First, the introduction of \cite{Cho26} does assert that \cite{Cas17} applies to the primary Burniat surface; the assertion is made there for no particular polarization, without proof, and without the verification of $h^1(\OO_X(H))=0$ on which \cite{Cas17} rests --- and it is exactly that verification, for every polarization at once, that Corollary~\ref{cor:burniat-nonspecial} supplies. Second, the same conclusion can be reached from published results alone: \cite[Prop.~5.4]{Cho26} for the degree bound, \cite[Prop.~4.2]{Cho26} for the ampleness criterion, and Kodaira vanishing. Our route replaces the first by Lemma~\ref{lem:bpf-degree} and the second by Lemma~\ref{lem:mori}, which is why the argument given here is self-contained; the conclusion is the same either way. Third, assertion~(b) is already contained in \cite[Thm.~1.2]{Cas17} for a general choice of the $0$-cycle, a primary Burniat surface being neither a rational scroll nor $\PP^2$; the increment here is that stability holds for \emph{every} admissible cycle, which is what the absence of Ulrich line bundles buys.
\end{remark}

\begin{remark}\label{rem:wild-burniat}
On a primary Burniat surface no Ulrich line bundle exists for any polarization \cite[Thm.~5.5]{Cho26}, so in Step~1 of the proof of Theorem~\ref{thm:evenrank} only the stable branch occurs. Moreover $K_X^2=6$, so $w=H^2-2$, while $H^2\ge18$ for every polarization by the proof of Corollary~\ref{cor:burniat-nonspecial}. Hence $w\ge16$ for \emph{every} polarization, and Theorem~\ref{thm:evenrank} applies unconditionally.
\end{remark}

\subsection{Where the special ample classes live}\label{subsec:rays}
This subsection is a classification, not a step in the proof of Corollary~\ref{cor:burniat-nonspecial}: we confine the numerical classes of the special ample divisors on $X$ to three rays, show that each ray does carry special classes, and compute their cohomology. The confinement is sharp as a necessary condition but is not a characterization --- the rays also carry non-special classes, as Remark~\ref{subsubsec:cho-rays} records. It proves Theorem~\ref{thm:burniat}(b), it gives a second and independent route to Corollary~\ref{cor:burniat-nonspecial}, and it sets up the eigensheaf technique used in \S\ref{subsec:twisted}. Overlaps with \cite{Cho26} are recorded at each statement.

\subsubsection{Reduction to three rays}
\begin{lemma}\label{lem:dichotomy}
Let $H$ be an ample divisor on $X$ with $h^1(\OO_X(H))\neq0$. Then $H-K_X$ is nef with $(H-K_X)^2=0$, and it lies on one of the three rays $\RR_{\ge0}(h-f_i)$. Equivalently,
\[
H\;\equiv\;K_X+c\,(h-f_i)\qquad\text{for some }i\in\{1,2,3\}\text{ and some integer }c\ge0 .
\]
Conversely every such class is ample, $K_X$ being ample and $h-f_i=\tfrac12\pi^*(e_0-e_i)$ nef.
\end{lemma}

\begin{proof}
By Lemma~\ref{lem:mori} it suffices to test $H-K_X$ on the six classes $[\varepsilon_i],[\sigma_{ij}]$, all of which satisfy $K_X\cdot C=1$; ampleness of $H$ gives $(H-K_X)\cdot C\ge0$ on each, so $H-K_X$ is nef. It is not big: otherwise Kawamata--Viehweg would give $h^1(\OO_X(K_X+(H-K_X)))=0$. A nef class is big exactly when its square is positive, so $(H-K_X)^2=0$.

It remains to identify the non-zero nef classes of square zero, the case $H\equiv K_X$ being the vertex $c=0$. Dualizing $\NE(X)=\iota(\NE(Y))$ and using that $\iota$ is an isometry gives $\mathrm{Nef}(X)=\iota(\mathrm{Nef}(Y))$, so $N:=\iota^{-1}(H-K_X)$ is a non-zero nef class of square zero on the del Pezzo surface $Y$. Such an $N$ is a positive multiple of a conic class. Indeed $K_Y-N$ is not effective, $-K_Y$ being nef and $(K_Y-N)\cdot(-K_Y)\le-K_Y^2<0$, so $h^2(\OO_Y(N))=0$ and $h^0(\OO_Y(N))\ge\chi(\OO_Y(N))=1+\tfrac12(-K_Y)\cdot N\ge1$; moreover $(-K_Y)\cdot N>0$, since $N$ is effective and non-zero and $-K_Y$ is ample on the del Pezzo surface $Y$. Then $N-K_Y$ is nef and big, so the basepoint-free theorem makes $N$ semiample; the Stein factorisation of $|mN|$, $m\gg0$, maps $Y$ onto a curve, necessarily $\PP^1$, and $N=cF$ with $F$ the class of a fibre. Connectedness of the fibres gives $p_a(F)=1-\tfrac12(-K_Y)\cdot F\ge0$, and the same identity forces $(-K_Y)\cdot F$ even, hence $(-K_Y)\cdot F=2$: $F$ is a conic class. On the degree six del Pezzo the conic classes are exactly $e_0-e_1,e_0-e_2,e_0-e_3$, the three fibre classes of the hexagon. Finally $h-f_i=\iota(e_0-e_i)$ is primitive in $\Num(X)$ and $H-K_X$ is integral on its ray, so $c\in\ZZ_{\ge0}$.
\end{proof}

\begin{remark}[Comparison with Cho]\label{rem:dichotomy-cho}
Neither half of the lemma is new, and we claim novelty only for the range of the second. The passage from $h^1(\OO_X(H))\neq0$ to $(H-K_X)^2=0$ also follows in two lines from \cite{Cho26}: an ample $H$ has $(H-K_X)\cdot Z_i\ge0$ on the six ramification curves, so $H-K_X$ is nef by \cite[Prop.~4.2]{Cho26} and $(H-K_X)^2\ge0$, while $(H-K_X)^2>0$ would give $h^1(\OO_X(H))=0$ by the vanishing step of \cite[Algorithm~4.18(3.1)]{Cho26}. The identification of the ample classes with $(H-K_X)^2=0$ as the three $K_X$-translated rays is \cite[Prop.~4.10]{Cho26}, proved there with the effective semigroup instead of by descending to $Y$ --- but under the additional hypothesis $e([D])=64$, that all $64$ torsion twists in the numerical class be effective. Lemma~\ref{lem:dichotomy} carries no such hypothesis and therefore also covers the vertex $c=0$, where $e([K_X])=63$: indeed $h^0(\OO_X(K_X))=p_g=0$, while $h^0(\OO_X(K_X+\tau))\ge1$ for each of the $63$ non-trivial torsion classes \cite[\S4.2]{Cho26}. By Remark~\ref{rem:vertex} the vertex is exactly where the cheapest special ample classes live, so the difference is not idle.
\end{remark}

\begin{remark}[The vertex is not empty]\label{rem:vertex}
Lemma~\ref{lem:dichotomy} constrains only the \emph{numerical} class of $H$, so $c=0$ means $H=K_X+\tau$ with $\tau$ torsion. Such an $H$ is ample, ampleness being a numerical condition, and it need not be non-special: for $\tau$ of order two, $\chi(\OO_X(K_X+\tau))=1$ and $h^2(\OO_X(K_X+\tau))=h^0(\OO_X(-\tau))=0$ give $h^0(\OO_X(K_X+\tau))=1+h^1(\OO_X(\tau))$. Such twists do occur, and \cite[Prop.~4.7]{Cho26} says which: $h^0(\OO_X(K_X+\tau))=2$ for three of the $63$ non-trivial torsion classes and $=1$ for the other sixty, so exactly three of the twists $K_X+\tau$ are special, each with $h^1=1$. The total agrees with an a priori count from the Inoue description: by \cite{BC11} primary Burniat surfaces are Inoue's surfaces \cite{Ino94}, $X=\widehat Z/(\ZZ/2)^3$ with $\widehat Z\subset E_1\times E_2\times E_3$ an invariant $(2,2,2)$-hypersurface and the action free, so decomposing $\varpi_*\OO_{\widehat Z}$ into characters and using $q(\widehat Z)=3$ gives $\sum_{\chi\neq0}h^1(\OO_X(\tau_\chi))=q(\widehat Z)-q(X)=3$; the three special classes lie necessarily among these seven $\tau_\chi$. The two counts are obtained by unrelated methods --- the Inoue description on one side, the algorithm of \cite{Cho26} on the other --- so their agreement is an independent confirmation of both.
\end{remark}

\subsubsection{Cohomology along the rays}
Write $G=(\ZZ/2)^2$ with non-trivial characters $\chi_1,\chi_2,\chi_3$ and let $L_i=L_{\chi_i}$ be the building data of the cover, determined by $2L_i\equiv\Delta_j+\Delta_k$ for $\{i,j,k\}=\{1,2,3\}$. In the labelling of \S\ref{subsec:burniat} the branch divisors are, with indices mod $3$,
\[
\Delta_i\;=\;\widetilde{s_i}+\widetilde{m_i^{(1)}}+\widetilde{m_i^{(2)}}+E_{i+2}
\;\equiv\;3e_0-3e_i-e_{i+1}+e_{i+2},
\]
where $s_i$ is the side through $p_i,p_{i+1}$ and $m_i^{(1)},m_i^{(2)}$ are the extra lines through $p_i$; hence
\[
L_1=3e_0-e_2-2e_3,\qquad L_2=3e_0-2e_1-e_3,\qquad L_3=3e_0-e_1-2e_2 .
\]
One checks the standard consistency relations $\sum\Delta_i\equiv-3K_Y$, $L_i+L_j\equiv L_k+\Delta_k$, $\sum_iL_i\cdot(L_i+K_Y)=-6$, whence $\chi(\OO_X)=4\chi(\OO_Y)+\tfrac12\sum L_i(L_i+K_Y)=1$ and $K_X^2=(2K_Y+\sum\Delta_i)^2=K_Y^2=6$. For divisors of the form $K_X+\pi^*F$ the identity $\pi_*\omega_X=\bigoplus_\chi\omega_Y\otimes L_\chi$ and the projection formula give
\begin{equation}\label{eq:eigen}
H^q\bigl(X,\,K_X+\pi^*F\bigr)\;=\;\bigoplus_{\chi}H^q\bigl(Y,\,K_Y+L_\chi+F\bigr),
\qquad L_{\chi_0}=0 ,
\end{equation}
reducing everything to the rational surface $Y$, where a \emph{nef} divisor $N$ has $h^1(N)=h^2(N)=0$ and $h^0(N)=\chi(N)$, since $N-K_Y$ is ample and $K_Y-N$ is never effective.

\begin{proposition}[cf.\ {\cite[Prop.~4.27]{Cho26}}]\label{prop:ray}
Let $u\ge1$ and set $H_u:=K_X+\pi^*\OO_Y\bigl(u(e_0-e_1)\bigr)$, of numerical class $K_X+2u(h-f_1)$. Then $H_u$ is ample and
\[
h^0(\OO_X(H_u))=3u,\qquad h^1(\OO_X(H_u))=u-1,\qquad h^2(\OO_X(H_u))=0 .
\]
In particular $H_u$ is special exactly for $u\ge2$; the class $H_1$ lies on the ray but is not special.
\end{proposition}

\begin{proof}
Apply \eqref{eq:eigen} with $F=u(e_0-e_1)$; the four classes on $Y$ are
\[
M_0=(u-3)e_0+(1-u)e_1+e_2+e_3,\qquad M_1=ue_0+(1-u)e_1-e_3,
\]
\[
M_2=ue_0-(1+u)e_1+e_2,\qquad M_3=ue_0-ue_1-e_2+e_3,
\]
and in each case $K_Y-M_j$ has negative degree, so $h^2(M_j)=0$ and $h^1$ is read off from $\chi$ once $h^0$ is known; recall that an irreducible curve $E$ with $M\cdot E<0$ is a fixed component, so that removing it preserves $h^0$ (but neither $h^1$ nor $h^2$). For $M_0$, discarding the fixed parts $e_2,e_3$ leaves plane curves of degree $u-3$ with a point of multiplicity $u-1>u-3$: none, so $h^0(M_0)=0$ and $\chi(M_0)=1-u$ gives $h^1(M_0)=u-1$. For $M_1$, the class $N=ue_0-(u-1)e_1$ is nef with $h^0(N)=\chi(N)=2u+1$, and imposing the base point $p_3$ is a non-trivial condition, so $h^0(M_1)=2u=\chi(M_1)$ and $h^1(M_1)=0$. For $M_2$, one has $M_2\cdot e_2<0$, so $E_2$ is a fixed component; after removing it the sections would be degree-$u$ plane curves with a point of multiplicity $u+1$, of which there are none, so $h^0(M_2)=0=\chi(M_2)$. For $M_3$, one has $M_3\cdot e_3<0$ and, after removing $E_3$, the $(-1)$-curve $\widetilde{s_1}=e_0-e_1-e_2$ has negative intersection too; peeling both off leaves $(u-1)(e_0-e_1)$, nef with $h^0=\chi=u$, so $h^0(M_3)=u=\chi(M_3)$ and $h^1(M_3)=0$. Summing, $h^\bullet=(3u,u-1,0)$; as a check, $\chi(\OO_X(H_u))=2u+1=1+\tfrac12H_u\cdot(H_u-K_X)$. The corresponding classes on the two other rays are obtained by the symmetry of the Burniat configuration permuting $p_1,p_2,p_3$, so all three rays behave identically.
\end{proof}

\begin{corollary}\label{cor:candidate}
The class $H:=H_2=K_X+\pi^*\OO_Y(2e_0-2e_1)\equiv7h-5f_1-f_2-f_3$ is ample with
\begin{gather*}
h^0(\OO_X(H))=6,\qquad h^1(\OO_X(H))=1,\qquad \chi(\OO_X(H))=5,\\
H^2=22,\qquad H\cdot K_X=14,\\
H\cdot\varepsilon_1=H\cdot\sigma_{23}=5,\qquad
H\cdot\lambda_1=2,\qquad H\cdot\lambda_2=H\cdot\lambda_3=6,\\
H\cdot\varepsilon_2=H\cdot\varepsilon_3=H\cdot\sigma_{12}=H\cdot\sigma_{13}=1 .
\end{gather*}
It is an explicit special ample class with $H-K_X\neq0$, of Ulrich length $\len=\chi(\OO_X(H))+1=6$; the cheapest special ample line bundles on $X$ are the torsion twists of Remark~\ref{rem:vertex}, which however carry no information about the rays.
\end{corollary}

\begin{proof}
Immediate from Proposition~\ref{prop:ray} with $u=2$ and the intersection numbers of Lemma~\ref{lem:mori}.
\end{proof}

\begin{remark}[Odd $c$, after Cho]\label{subsubsec:cho-rays}
Formula \eqref{eq:eigen} applies to $K_X+\pi^*F$, hence to even $c$ only. The remaining classes are given in closed form by \cite[Prop.~4.27]{Cho26}, whose family $[2\ell;0,0,\ell]+[K_X]$ is exactly our ray with $\ell=c$: there $h^0$ is computed for every $c\ge1$ and each of the $64$ torsion classes, the untwisted case reading $h^0=\lfloor(3c+1)/2\rfloor$ and, since $\chi=c+1$, $h^1=\lfloor(c-1)/2\rfloor$. At $c=2u$ this is $(3u,u-1,0)$, in agreement with Proposition~\ref{prop:ray}; the agreement also fixes the identification of normalizations, and, the two computations being independent --- eigensheaf descent to $Y$ here, the symmetric-coordinate algorithm of \cite{Cho26} there --- it is a genuine cross-check of both. In particular the rays carry special classes for odd $c$ as well, already at $c=3$ where $h^0=5$ against $\chi=4$, while at $c=2$ one has $h^0=3=\chi$, so that class is not special.
\end{remark}

\subsubsection{No class on the rays is base point free}
The four degrees equal to one in Corollary~\ref{cor:candidate} are what obstructs base point freeness, through an elementary fact not particular to Burniat surfaces.

\begin{proposition}\label{prop:notva}
Let $H=K_X+c\,(h-f_i)$, $c\ge0$, be an integral class on one of the three rays of Lemma~\ref{lem:dichotomy}, the vertex $c=0$ included. Then $H$ is ample but \emph{not} base point free; in particular it is not very ample. Only intersection numbers and the genera of the ramification curves enter, so the same holds for every line bundle numerically equivalent to $H$, in particular for all torsion twists.
\end{proposition}

\begin{proof}
Fix $j\neq i$ and consider $\varepsilon_j=R(E_j)$. In $\Num(X)$ one has $h\cdot f_j=f_i\cdot f_j=0$, so $(h-f_i)\cdot\varepsilon_j=0$ and therefore $H\cdot\varepsilon_j=K_X\cdot\varepsilon_j=1$ for every $c\ge0$ --- at the vertex this reads $H\equiv K_X$ and applies to all torsion twists $K_X+\tau$, precisely the classes shown to be special in Remark~\ref{rem:vertex}. Now $\varepsilon_j$ is the double cover of $E_j\cong\PP^1$ branched at the four points where $E_j$ meets the branch components not containing it, hence a smooth elliptic curve, in accordance with $p_a(\varepsilon_j)=1$ in the table of \S\ref{subsec:burniat}. Lemma~\ref{lem:bpf-degree} forbids $H$ from being base point free, and a very ample divisor is base point free.
\end{proof}

Together with Lemma~\ref{lem:dichotomy}, Proposition~\ref{prop:ray} and Remark~\ref{rem:vertex} this proves Theorem~\ref{thm:burniat}(b), and gives the announced second route to Corollary~\ref{cor:burniat-nonspecial}: an ample and base point free $H$ with $h^1\neq0$ would lie on a ray, and no class on a ray is base point free.

\begin{remark}[Attribution]\label{rem:cho-notva}
Proposition~\ref{prop:notva} is the special case of Corollary~\ref{cor:degreetwo} in which the ample class lies on one of the rays, and it is proved by the same lemma; we have stated it separately only because the degree computation is immediate there from the lattice description. The same statement is contained in \cite[Prop.~5.4]{Cho26}, in the stronger form quantified over all ample and base point free divisors --- see Remark~\ref{rem:cho-degreetwo}.
\end{remark}

\begin{remark}[Why the rays are invisible to the standard tools]\label{rem:reider-silent}
Reider's criterion for very ampleness of $K_X+L$ requires $L^2\ge10$, while here $L=H-K_X=c(h-f_i)$ has $L^2=0$ for every $c$: the entire ray lies outside its reach, and Proposition~\ref{prop:notva} shows that this silence is not an artefact of the method --- not merely very ampleness but already base point freeness genuinely fails. Special classes are confined to the thin locus where the standard positivity techniques go silent simultaneously, and on primary Burniat surfaces that locus contains no polarization at all. All cohomology groups appearing along the rays, including those of \S\ref{subsec:twisted}, are also accessible through \cite{Cho26}, in the closed form of \cite[Prop.~4.27]{Cho26} for the untwisted classes and algorithmically otherwise; the two agree where we have compared them.
\end{remark}

\begin{remark}[Cayley--Bacharach for free, and what it costs]\label{rem:CBfree}
Take $H=H_2$ and $\varepsilon\in\{\varepsilon_2,\varepsilon_3,\sigma_{12},\sigma_{13}\}$, so that $H\cdot\varepsilon=1$ and $\len(Z)=\tfrac12H\cdot(H-K_X)+H\cdot\varepsilon=5$. The Cayley--Bacharach system is $|H+2\varepsilon|$, and $(H+2\varepsilon)\cdot\varepsilon=H\cdot\varepsilon+2\varepsilon^2=-1<0$, so $\varepsilon$ is a \emph{fixed component} of it. Consequently, for any reduced $0$-cycle $Z\subset\varepsilon$, every divisor of the system through a colength-one subscheme of $Z$ contains $\varepsilon\supset Z$: the Cayley--Bacharach property holds automatically, with no genericity and no B\'ezout threshold.

There is a price, and it is structural. For $\varepsilon$ rigid the unique section of $\OO_X(\varepsilon)$ lies in $H^0(\OO_X(\varepsilon)\otimes\II_Z)$ as soon as $Z\subset\varepsilon$, so the $-2H$ vanishing of \S\ref{subsec:twisted} needs its own connecting map to be injective; choosing $Z\not\subset\varepsilon$ makes that vanishing automatic but destroys the free Cayley--Bacharach property. This is a second obstruction for twisted kernels, independent of $h^1$: \emph{the two Ulrich twists pull the cycle in opposite directions along $\varepsilon$}. (In \S\ref{subsec:pg} the opposite convention $Z\cap G=\emptyset$ is adopted, precisely to keep the second vanishing free.) What remains is the connecting-map condition of \S\ref{subsec:twisted}, and there the answer is negative.
\end{remark}

\begin{remark}[Why the primary case]\label{rem:integrality}
The reduction of this subsection uses that $\iota$ carries $\Num(Y)$ integrally into $\Num(X)$: only then is $H\cdot\iota(\Gamma)$ a positive integer on every extremal class, so that $(H-K_X)\cdot\iota(\Gamma)\ge0$ follows from ampleness together with $K_X\cdot\iota(\Gamma)=1$. On a primary Burniat surface this holds because the six $(-1)$-curves of the degree six del Pezzo are precisely $E_1,E_2,E_3$ and the three sides, all components of the branch configuration, so $\pi^*\Gamma=2R(\Gamma)$ on every extremal ray. For $K_X^2<6$ the base is a blow up of $\PP^2$ at $9-K_X^2\ge4$ points, its extremal rays need no longer be exhausted by branch components, and for $\Gamma\not\subset\Delta$ the divisor $\pi^*\Gamma$ is reduced, so $\iota(\Gamma)$ need not be integral. We have not analysed the degenerate configurations, and we do not know whether every very ample divisor on a Burniat surface with $K_X^2<6$ is non-special.
\end{remark}

\subsection{Twisted kernels: the surviving mechanism}\label{subsec:twisted}
Theorem~\ref{thm:nogo} localizes the obstruction: the connecting homomorphism $H^0(\OO_X(H)\otimes\II_Z)\to H^1(\OO_X(K_X))$ has zero target when $q=0$, so the unwanted sections inject into $H^0(\EE(-H))$ and must be killed at the source, which is impossible below the independence threshold. The way out is to \emph{move the kernel}, so that the connecting map acquires a non-zero target and the extension class can do the killing.

Concretely, let $A=H+K_X-\varepsilon$ and $B=2H+\varepsilon$ for a curve class $\varepsilon$, so that $c_1=A+B=3H+K_X$ is still special. Twisting $0\to\OO_X(A)\to\EE\to\OO_X(B)\otimes\II_Z\to0$ by $-H$ gives the exact sequence
\begin{multline*}
0\to H^0(\OO_X(K_X-\varepsilon))\to H^0(\EE(-H))\\
\to H^0(\OO_X(H+\varepsilon)\otimes\II_Z)
\xrightarrow{\ \delta_e\ } H^1(\OO_X(K_X-\varepsilon)),
\end{multline*}
where $\delta_e$ is cup product with the extension class. Thus
\[
h^0\bigl(\EE(-H)\bigr)=0
\iff
h^0(\OO_X(K_X-\varepsilon))=0\ \text{and}\ \delta_e\ \text{is injective}.
\]
The twist by $-2H$ behaves differently, and the difference matters. It reads
\[
0\to\OO_X(K_X-\varepsilon-H)\to\EE(-2H)\to\OO_X(\varepsilon)\otimes\II_Z\to0 ,
\]
so that $h^0(\EE(-2H))=0$ requires, besides $h^0(\OO_X(K_X-\varepsilon-H))=0$ (automatic for $p_g=0$ and $H$ ample), that the corresponding connecting map be injective on $H^0(\OO_X(\varepsilon)\otimes\II_Z)$. Speciality of $H$ can thus be \emph{absorbed} by the extension class, up to the dimension of the target --- but on two twists at once, and the two impose opposite demands on $Z$; see Remark~\ref{rem:CBfree}.

\smallskip
\noindent \emph{First numerical steps on a primary Burniat surface.} Take $\varepsilon=\varepsilon_1$ (class $f_1$, so $\varepsilon^2=-1$, $K_X\cdot\varepsilon=1$, $p_a=1$) and $A=H+K_X-\varepsilon$, $B=2H+\varepsilon$. One computes:
\begin{enumerate}[label=\textup{(\alph*)}]
\item \emph{Kernel vanishings:} $h^0(\OO_X(A-H))=h^0(\OO_X(K_X-\varepsilon))=0$ and $h^0(\OO_X(A-2H))=0$, both because $p_g=0$.
\item \emph{The absorber:} $\chi(\OO_X(K_X-\varepsilon))=1+\tfrac12(K_X-\varepsilon)\cdot(-\varepsilon)=0$, $h^0=0$, and $h^2(\OO_X(K_X-\varepsilon))=h^0(\OO_X(\varepsilon))=1$ (the curve $\varepsilon$ is rigid: $\varepsilon^2=-1$). Hence
\[
h^1\bigl(\OO_X(K_X-\varepsilon)\bigr)=1 :
\]
the connecting map $\delta_e$ has a one-dimensional target --- exactly the room needed to absorb one unit of speciality.
\item \emph{The cycle:} $\len(Z)=c_2-A\cdot B=\tfrac12H\cdot(H-K_X)+H\cdot\varepsilon$ \textup{(}using $\varepsilon^2-K_X\cdot\varepsilon+2=0$, valid for all twelve curves of the table; for general $\varepsilon$ the right-hand side acquires the extra term $\varepsilon^2-K_X\cdot\varepsilon+2$\textup{)}, and one checks $\chi(\OO_X(H+\varepsilon))=\len(Z)$ \textup{(}by the same identity\textup{)}; since $h^2(\OO_X(H+\varepsilon))=h^0(\OO_X(K_X-H-\varepsilon))=0$ for $p_g=0$ and $H$ ample, this gives
\[
h^0\bigl(\OO_X(H+\varepsilon)\otimes\II_Z\bigr)\;\ge\;h^1\bigl(\OO_X(H+\varepsilon)\bigr):
\]
the speciality budget has moved from $h^1(\OO_X(H))$ to $h^1(\OO_X(H+\varepsilon))$, and the construction can close whenever $h^1(\OO_X(H+\varepsilon))\le1$, $Z$ is chosen so that equality holds in the display, and $e$ is chosen with $\delta_e\neq0$.
\item \emph{Local freeness:} the Cayley--Bacharach system is now $|B-A+K_X|=|H+2\varepsilon|\neq|H|$, so Theorem~\ref{thm:nogo} does not apply; this is precisely where the points-on-a-curve mechanism of Lemma~\ref{lem:bezout} (with $C$ among the curves of the table in \S\ref{subsec:burniat}, e.g.,\ $C=\varepsilon$ itself, for which $(H+2\varepsilon)\cdot\varepsilon=H\cdot\varepsilon-2$ is small) re-enters the picture.
\end{enumerate}
\smallskip
\noindent \emph{The Burniat test for twisted kernels.} By Corollary~\ref{cor:candidate} an explicit ample class with $h^1(\OO_X(H))=1$ is available, namely $H=H_2$, and by Remark~\ref{rem:CBfree} the Cayley--Bacharach step is then automatic for cycles on any of $\varepsilon_2,\varepsilon_3,\sigma_{12},\sigma_{13}$. By Proposition~\ref{prop:notva} this $H$ is not base point free, hence not a polarization; but the vanishings of Definition~\ref{def:ulrich} are purely cohomological and still make sense, and it is instructive that the construction fails there for a second reason, independent of the failure of positivity. Proposition~\ref{prop:twistcoh} and Theorem~\ref{thm:twistedfail} together prove part (c) of Theorem~\ref{thm:burniat}.

\begin{lemma}[Ramification-twisted eigensheaves, cf.\ \cite{Par91}]\label{lem:ramtwist}
Let $\Gamma\subset\Delta_i$ be a reduced component of the branch divisor of the bidouble cover $\pi:X\to Y$, let $\varepsilon=R(\Gamma)$ be its reduced preimage, and let $g_i\in G$ be the inertia element along $\Gamma$. Then
\[
\pi_*\omega_X(\varepsilon)\;\cong\;
\bigoplus_{\chi(g_i)=1}\omega_Y\otimes L_\chi(\Gamma)
\;\oplus\;
\bigoplus_{\chi(g_i)=-1}\omega_Y\otimes L_\chi .
\]
With the conventions of \S\ref{subsec:rays}, where $2L_i\equiv\Delta_j+\Delta_k$, the characters trivial on $g_i$ are $\chi_0$ and $\chi_i$: in the decomposition \eqref{eq:eigen} of $H^\bullet(X,K_X+\pi^*F)$ the summands $M_0$ and $M_i$ acquire the extra divisor $\Gamma$, while $M_j$ and $M_k$ are unchanged.
\end{lemma}

\begin{proof}
The statement is local at the generic point of $\Gamma$, where the cover is the product of an \'etale double cover with the double cover $x=t^2$, and $\varepsilon=\{t=0\}$. For the ramified factor, $\OO_X(-\varepsilon)=(t)=t\,\OO_Y\oplus x\,\OO_Y$ as an $\OO_Y$-module, the second summand being the invariant part; hence
\[
\pi_*\OO_X(-\varepsilon)\;=\;\bigoplus_{\chi(g_i)=1}L_\chi^{-1}(-\Gamma)\;\oplus\;\bigoplus_{\chi(g_i)=-1}L_\chi^{-1} .
\]
Both sides of this display are direct summands of $\pi_*\OO_X(-\varepsilon)$ as $\OO_Y$-modules, and $\pi_*\OO_X(-\varepsilon)$ is locally free of rank four, $\pi$ being finite and flat and $\OO_X(-\varepsilon)$ a line bundle; a direct summand of a locally free sheaf is locally free, so each character eigensheaf is a line bundle on the smooth surface $Y$. Off the branch locus $\pi$ is \'etale and the decomposition is the standard one of \cite{Par91}, while the local model above identifies the two sides at every point of a branch component lying on no other one; the two sides therefore agree outside the finite set where distinct branch components meet. Two line bundles on a smooth surface which agree in codimension one are isomorphic, a divisor class supported in codimension two being trivial, so the display holds globally. Relative duality gives $\pi_*\omega_X(\varepsilon)=\mathcal{H}om_{\OO_Y}\bigl(\pi_*\OO_X(-\varepsilon),\omega_Y\bigr)$, which is the displayed decomposition. For the last assertion, the fundamental relation $2L_\chi\equiv\sum\{\Delta_m:\chi(g_m)=-1\}$ specializes to $2L_i\equiv\Delta_j+\Delta_k$ precisely when $\chi_i(g_i)=1$ and $\chi_i(g_j)=\chi_i(g_k)=-1$.
\end{proof}

\begin{proposition}\label{prop:twistcoh}
Let $H_u=K_X+\pi^*\OO_Y\bigl(u(e_0-e_1)\bigr)$ with $u\ge1$, and let $\varepsilon$ be one of the four curves
\[
\varepsilon_2,\qquad \varepsilon_3,\qquad \sigma_{12},\qquad \sigma_{13}
\]
of the table in \S\ref{subsec:burniat}, all of which satisfy $H_u\cdot\varepsilon=1$. Then
\[
h^\bullet\bigl(\OO_X(H_u+\varepsilon)\bigr)\;=\;(3u+1,\;u,\;0) .
\]
In particular $h^\bullet(\OO_X(H_2+\varepsilon))=(7,2,0)$.
\end{proposition}

\begin{proof}
By the labelling of \S\ref{subsec:rays}, $\varepsilon_3=R(E_3)$ and $\sigma_{12}=R(\widetilde{s_1})$ are components of $\Delta_1$, while $\varepsilon_2=R(E_2)$ and $\sigma_{13}=R(\widetilde{s_3})$ are components of $\Delta_3$. By Lemma~\ref{lem:ramtwist} the twist adds $\Gamma$ to $M_0$ and $M_1$ in the first case, to $M_0$ and $M_3$ in the second; the remaining two summands keep the cohomology computed in Proposition~\ref{prop:ray}.

In all four cases the modified $M_0$ summand has $h^\bullet=(0,u,0)$. Indeed, after peeling off the fixed exceptional components, $M_0+\Gamma$ reduces to $(u-3)e_0-(u-1)e_1$ when $\Gamma\in\{E_2,E_3\}$ and to $(u-2)e_0-ue_1$ when $\Gamma\in\{\widetilde{s_1},\widetilde{s_3}\}$; in both cases the imposed multiplicity at $p_1$ exceeds the plane degree by two, so $h^0=0$. Moreover $K_Y-(M_0+\Gamma)=-u(e_0-e_1)-\Gamma$ has degree $-2u-1<0$ against $-K_Y$, so $h^2=0$, and Riemann--Roch gives $h^1=u$.

The modified effective summands are as follows; recall that a nef divisor $N$ on $Y$ has $h^0(N)=\chi(N)$.
\begin{itemize}[leftmargin=1.4em]
\item $\varepsilon=\varepsilon_3$: $M_1+e_3=ue_0-(u-1)e_1$ is nef with $\chi=2u+1$. Adding the unchanged $M_3$, which contributes $u$, gives $h^0=3u+1$.
\item $\varepsilon=\sigma_{12}$: $M_1+(e_0-e_1-e_2)=(u+1)e_0-ue_1-e_2-e_3$ is nef with $\chi=2u+1$; again $M_3$ contributes $u$.
\item $\varepsilon=\varepsilon_2$: $M_3+e_2=u(e_0-e_1)+e_3$; here $E_3$ is a fixed component, and removing it leaves the nef class $u(e_0-e_1)$ with $\chi=u+1$. The unchanged $M_1$ contributes $2u$.
\item $\varepsilon=\sigma_{13}$: $M_3+(e_0-e_1-e_3)=(u+1)e_0-(u+1)e_1-e_2$; here the $(-1)$-curve $e_0-e_1-e_2$ is fixed, and removing it again leaves $u(e_0-e_1)$, so $h^0=u+1$. The unchanged $M_1$ contributes $2u$.
\end{itemize}
Thus $h^0=3u+1$ in every case. The vanishing $h^2=0$ has to be checked separately, since the peeling above preserves only $h^0$. For the four effective modified summands $M$ displayed in the bullets one has $(-K_Y)\cdot M=2u+1>0$, so $(-K_Y)\cdot(K_Y-M)=-6-(2u+1)<0$ and $K_Y-M$ is not effective, $-K_Y$ being nef; hence $h^2(M)=h^0(K_Y-M)=0$; likewise for the two unchanged summands by Proposition~\ref{prop:ray}, and for the modified $M_0$-summand, which is \emph{not} effective, one computes $(-K_Y)\cdot M_0=2u-6$ and $(-K_Y)\cdot\Gamma=1$, $\Gamma$ being a $(-1)$-curve, so that $(-K_Y)\cdot(K_Y-M_0-\Gamma)=-6-(2u-5)=-2u-1<0$ and $K_Y-M_0-\Gamma$ is not effective either, $-K_Y$ being ample. Finally, using $(h-f_1)\cdot\varepsilon=0$, $\varepsilon^2=-1$, $K_X\cdot\varepsilon=1$ and $K_X\cdot(h-f_1)=2$,
\[
\chi\bigl(\OO_X(H_u+\varepsilon)\bigr)=1+\tfrac12(H_u+\varepsilon)\cdot(H_u+\varepsilon-K_X)=2u+1 ,
\]
whence $h^1=(3u+1)-(2u+1)=u$.
\end{proof}

\begin{theorem}\label{thm:twistedfail}
Let $X$ be a primary Burniat surface, let $u\ge2$, $H=H_u$, and let $\varepsilon\in\{\varepsilon_2,\varepsilon_3,\sigma_{12},\sigma_{13}\}$. Then no extension
\[
0\longrightarrow\OO_X(H+K_X-\varepsilon)\longrightarrow\EE\longrightarrow\OO_X(2H+\varepsilon)\otimes\II_Z\longrightarrow0
\]
carrying the special Ulrich Chern classes satisfies $h^0(\EE(-H))=0$, for any $0$-dimensional $Z$ and any extension class $e$. In particular no such $\EE$ satisfies the Ulrich vanishings.
\end{theorem}

\begin{proof}
Items (b) and (c) above apply verbatim to each of the four curves, all of which satisfy $\varepsilon^2=-1$ and $K_X\cdot\varepsilon=1$. No local freeness of $\EE$ is needed anywhere below: Chern classes are additive on an extension of coherent sheaves in any case, so the computation of $\len(Z)$ from $c_2$ is unaffected. By (c),
\[
\len(Z)=\tfrac12H\cdot(H-K_X)+H\cdot\varepsilon=2u+1 ,
\]
using $H-K_X=2u(h-f_1)$, $(h-f_1)^2=0$, $K_X\cdot(h-f_1)=2$ and $H\cdot\varepsilon=1$. A $0$-dimensional subscheme of length $2u+1$ imposes at most $2u+1$ conditions on $|H+\varepsilon|$, so Proposition~\ref{prop:twistcoh} gives
\[
h^0\bigl(\OO_X(H+\varepsilon)\otimes\II_Z\bigr)\;\ge\;(3u+1)-(2u+1)\;=\;u\;\ge\;2 .
\]
By (b) the target $H^1(\OO_X(K_X-\varepsilon))$ of $\delta_e$ is one-dimensional, so $\delta_e$ cannot be injective; by the displayed exact sequence, $h^0(\EE(-H))\neq0$.
\end{proof}

\begin{remark}\label{rem:twistedfail-scope}
The obstruction is sharp in $u$ and in $\varepsilon$. For $u=1$ the count only gives $h^0(\OO_X(H_1+\varepsilon)\otimes\II_Z)\ge1$, compatible with a one-dimensional absorber --- but $h^1(\OO_X(H_1))=0$ by Proposition~\ref{prop:ray}, so no twisted kernel is needed there. The four curves covered are exactly those of $H_u$-degree one, which are exactly those for which the Cayley--Bacharach step is free (Remark~\ref{rem:CBfree}); for $\varepsilon_1$ one has $(H_u+2\varepsilon_1)\cdot\varepsilon_1=2u-1>0$, so $\varepsilon_1$ is not a fixed component of the Cayley--Bacharach system and is not covered, and neither are twists by classes outside the table. We stress that by Corollary~\ref{cor:burniat-nonspecial} none of this could produce an Ulrich bundle with respect to a genuine polarization: the classes involved are ample but never base point free.
\end{remark}

\begin{question}\label{q:twisted}
Let $X$ be a surface with $p_g=q=0$ and $H$ very ample with $h^1(\OO_X(H))=s>0$. Do there always exist a curve class $\varepsilon$ and a pair $(Z,e)$ as above, with $Z$ reduced (so that Proposition~\ref{prop:CB} applies) --- with $h^0(\OO_X(K_X-\varepsilon))=0$, $\dim H^1(\OO_X(K_X-\varepsilon))\ge s$, and $\delta_e$ injective on both twists --- producing a special Ulrich bundle with kernel $\OO_X(H+K_X-\varepsilon)$?
\end{question}

\section{Multiples of a polarization and the Coskun--Huizenga theorem}\label{sec:asymptotic}

Throughout this section $H$ is very ample and we consider the polarizations $mH$, $m\ge1$. The target statement is \cite[Thm.~1.2]{CH20}: on \emph{every} smooth projective surface, rank two Ulrich bundles exist on $(X,mH)$ for $m\gg0$. We record what the elementary constructions of \S\ref{sec:methods} give, and where they stop.

\subsection{The regular case}\label{subsec:asymptotic-reg}

\begin{proposition}\label{prop:asymptotic}
Let $X$ be a smooth projective surface with $p_g=q=0$ and let $H$ be very ample. Then for every $m\ge1$ with $mH-K_X$ nef and big, the pair $(X,mH)$ carries a $\mu_{mH}$-semistable special Ulrich bundle of rank two, stable unless $X$ carries an Ulrich line bundle for $mH$.
\end{proposition}

\begin{proof}
The divisor $mH=H+(m-1)H$ is very ample, being the sum of a very ample and a globally generated divisor, and $h^1(\OO_X(mH))=h^1(\OO_X(K_X+(mH-K_X)))=0$ by Kawamata--Viehweg. So $(X,mH)$ satisfies the hypotheses of Theorem~\ref{thm:main}; semistability and the stability criterion are Lemma~\ref{lem:semistable} and Corollary~\ref{cor:stable}.
\end{proof}

The statement in this range is due to Beauville \cite{Bea16b}; what the elementary proof adds is the explicit threshold, often very small --- for $H=aK_X$ pluricanonical on a minimal surface of general type with $p_g=q=0$ \textup{(}$a\ge3$ on primary Burniat surfaces, Remark~\ref{rem:burniat-positive}\textup{)}, $mH-K_X=(ma-1)K_X$ is ample for every $m\ge1$, so all multiples work. Non-speciality is here \emph{produced} rather than assumed, Kawamata--Viehweg supplying it once $m$ clears the threshold; this is why the asymptotic problem is so much softer than the fixed-polarization one of \S\ref{subsec:rays}.

\subsection{Twisted kernels and the budget identity}\label{subsec:pg}
Outside the regular case the adjoint kernel no longer suffices: for $p_g>0$ it fails at the outset, $H^0(\OO_X(K_X))$ injecting into $H^0(\EE(-mH))$, and for $p_g=0<q$ it survives with a deficit. Both are instances of one configuration, set up once below. All $0$-cycles are assumed \emph{reduced}, so that ``every point is redundant'' is meaningful and Proposition~\ref{prop:CB} applies.

Fix an effective divisor $G\ge0$ with $h^0(\OO_X(K_X-G))=0$ --- for $p_g=0$ one may take $G=0$, and for $p_g>0$ any general $G\in|kH|$ with $kH^2>K_X\cdot H$ --- set $s:=h^1(\OO_X(K_X-G))$ and
\begin{gather*}
A:=mH+K_X-G,\qquad B:=2mH+G,\\
V:=B-mH=mH+G,\qquad M:=B-A+K_X=mH+2G ,
\end{gather*}
so that $c_1=A+B=3mH+K_X$ is special, $|V|$ is the system controlling $h^0(\EE(-mH))$ and $|M|$ is the Cayley--Bacharach system. Assume $Z\cap G=\emptyset$ and $h^0(\OO_X(G)\otimes\II_Z)=0$, the second being all that the twist by $-2mH$ requires. Both are open conditions on reduced cycles of the length at issue, and both are satisfiable: a general such cycle avoids $G$ and imposes independent conditions on $|G|$, so the second holds as soon as $\len(Z)\ge h^0(\OO_X(G))$, which by Proposition~\ref{prop:budget} reads $h^0(\OO_X(V))\ge s$ and so holds for $m\gg0$. Then $h^0(\EE(-mH))=0$ if and only if the connecting map
\[
\delta_e:H^0\bigl(\OO_X(V)\otimes\II_Z\bigr)\longrightarrow H^1\bigl(\OO_X(K_X-G)\bigr)\cong\CC^{s}
\]
is injective. For $G=0$ this is the adjoint kernel itself, with $V=M=mH$ and $s=h^1(\OO_X(K_X))=q$.

\begin{proposition}[Budget identity]\label{prop:budget}
Keep the notation and the standing assumptions above: $G\ge0$ is effective with $h^0(\OO_X(K_X-G))=0$ and $s=h^1(\OO_X(K_X-G))$; the divisors $A,B,V,M$ are as displayed; $\EE$ is an extension of $\OO_X(B)\otimes\II_Z$ by $\OO_X(A)$ with $Z$ reduced of the Ulrich length, satisfying the Cayley--Bacharach property with respect to $|M|$; and $Z\cap G=\emptyset$ with $h^0(\OO_X(G)\otimes\II_Z)=0$. Let $m$ be large enough that $V-K_X$ is nef and big, so that
$h^1(\OO_X(V))=h^2(\OO_X(V))=0$. Then
\[
\len(Z)=h^0(\OO_X(V))+h^0(\OO_X(G))-s ,\quad
h^0\bigl(\OO_X(V)\otimes\II_Z\bigr)\ge 1-h^0(\OO_X(G))+s,
\]
the first from the Ulrich value of $c_2$, the second from the Cayley--Bacharach property. For $G$ rigid the two combine with $h^0(\OO_X(V)\otimes\II_Z)\le s$, forced by injectivity of $\delta_e$, to give
\[
h^0\bigl(\OO_X(V)\otimes\II_Z\bigr)=s,\quad c_V(Z)=\len(Z)-1,\quad
\delta_e:\CC^{s}\xrightarrow{\ \sim\ }\CC^{s},
\]
where $c_V(Z)$ is the number of conditions imposed on $|V|$; and these conditions are also sufficient. The counting obstructions thus close \emph{exactly at the boundary}, the necessary inequality being met with equality by the budget.
\end{proposition}

\begin{proof}
The budget: $\chi(\EE(-mH))=0$ gives $\len(Z)=\chi(\OO_X(A-mH))+\chi(\OO_X(B-mH))=\chi(\OO_X(K_X-G))+\chi(\OO_X(V))$, then $\chi(\OO_X(K_X-G))=h^0(\OO_X(G))-s$ by Serre duality and $\chi(\OO_X(V))=h^0(\OO_X(V))$ by the assumed vanishing. The inequality: if $x\in Z$ is redundant for $|M|$ --- every section of $\OO_X(M)\otimes\II_{Z'}$ vanishes at $x$, with $Z'$ the colength-one subscheme at $x$ --- then so does every section of $s_G\cdot H^0(\OO_X(V)\otimes\II_{Z'})$; as $s_G(x)\neq0$, the sections of $\OO_X(V)\otimes\II_{Z'}$ vanish at $x$ themselves, so $Z$ has the Cayley--Bacharach property for $|V|$ too, and step (4) of Theorem~\ref{thm:nogo} applied to $|V|$ gives the bound. For sufficiency, the Cayley--Bacharach property makes the extension locally free by Proposition~\ref{prop:CB}, injectivity of $\delta_e$ gives $h^0(\EE(-mH))=0$, and the twist by $-2mH$ has quotient $\II_Z(G)$, without sections by hypothesis, and kernel $\OO_X(A-2mH)=\OO_X(K_X-G-mH)$, also without sections, since multiplication by a non-zero section of $\OO_X(mH)$ embeds $H^0(\OO_X(K_X-G-mH))$ into $H^0(\OO_X(K_X-G))=0$; Proposition~\ref{prop:reduction} concludes.
\end{proof}

For $G=0$ and $p_g=0<q$ this reads: the construction closes precisely when $Z$ imposes $c(Z)=\len(Z)-1$ conditions on $|mH|$, every point being redundant, and $\delta_e:\CC^q\to\CC^q$ is an isomorphism. The deficit below the independence threshold is exactly $q$, matched by the absorber $H^1(\OO_X(K_X))$; the extension is then unique, Serre duality for Ext groups identifying
\[
\Ext^1\bigl(\OO_X(2mH)\otimes\II_Z,\,\OO_X(mH+K_X)\bigr)\;\cong\;H^1\bigl(\OO_X(mH)\otimes\II_Z\bigr)^\vee ,
\]
of dimension $\len(Z)-c(Z)=1$, so the condition is a property of $Z$ alone. The missing lemma is thus sharply isolated, and we have not settled it by the methods of \S\ref{sec:methods}.

For $p_g>0$ the smallest case is already instructive. On a smooth quintic $X\subset\PP^3$ one has $K_X=H$ and $H^2=5$, so $kH^2>K_X\cdot H$ first holds at $k=2$: take $G\in|2H|$ general, so that $A=0$, $V=3H$, $M=5H$. Arithmetic Cohen--Macaulayness gives $h^1(\OO_X(kH))=0$ for all $k$, whence $s=h^1(\OO_X(-H))=0$ and both Ulrich vanishings fall on the cycle; at $m=1$ the budget reads $\len(Z)=h^0(\OO_X(3H))+h^0(\OO_X(2H))=30=20+10$. This is an identity and not a pair of competing requirements, since multiplication by a section of $\OO_X(H)$ embeds $H^0(\OO_X(2H)\otimes\II_Z)$ into $H^0(\OO_X(3H)\otimes\II_Z)$. The binding condition is Cayley--Bacharach for $|5H|$: the $30$ points must impose at most $29$ conditions on $|5H|$, of dimension $55$, while imposing the maximum $20$ on $|3H|$ --- special position, positional and not numerical.

So no \emph{counting} obstruction survives; the difficulty is constructive --- and it is a difficulty about the method, not about the object, because the bundles are there. A rank two bundle on the quintic $X$ with $c_1=4H=3H+K_X$ and $c_2=30$ is Ulrich as soon as $h^0(\EE(-H))=0$, by Proposition~\ref{prop:reduction}; and such a bundle exists on a general quintic surface by \cite[Thm.~6.13]{CF09}, the existence having been obtained earlier by Beauville and Schreyer in the appendix to \cite{Bea00} through a computer algebra computation. The elementary machine of \S\ref{sec:methods} cannot see them, which is precisely the point: what it is missing is positional and not numerical. For generic $Z$, Cayley--Bacharach for $|M|$ needs $\len(Z)\ge h^0(\OO_X(M))+1$ by Lemma~\ref{lem:LQ25}, while $h^0(\OO_X(M))-\len(Z)$ grows like $m(G\cdot H)$: a deficit linear in $m$ survives, and the points-on-curves mechanism of \S\ref{subsec:Anghel} buys the property back at a cost again linear in $m$, so the two slopes must be compared architecture by architecture. What is needed is this.

\begin{question}\label{q:cycles}
With the notation above, do there exist, for $m\gg0$, reduced zero-cycles $Z$ of the exact Ulrich length which are \emph{Cayley--Bacharach-special} \textup{(}every point redundant for $|M|$\textup{)} and \emph{Brill--Noether-general} \textup{(}imposing $\len(Z)-1$ conditions on $|V|$, with the unique relation pairing non-trivially under $\delta_e$\textup{)}? An affirmative answer would give an elementary and effective proof of \cite[Thm.~1.2]{CH20}; this implication is unconditional. The converse is not: \cite{CH20} produces Ulrich bundles, and a rank two Ulrich bundle need not admit a sub-line bundle of the shape $\OO_X(A)$ at all, so the existence of the cycles cannot be read off from \cite{CH20}.
\end{question}

\section*{Declaration of generative AI and AI-assisted technologies}

During the preparation of this work, the authors used ChatGPT and Claude to improve the language and clarity of the manuscript and to assist in checking the mathematical arguments and proofs. After using these tools, the authors independently verified and reviewed all mathematical content and take full responsibility for the content of the published article.

\end{document}